\documentclass[a4paper,10pt]{article}
\usepackage[utf8x]{inputenc}
\usepackage{pstricks}
\usepackage{pst-node}
\usepackage{pst-eps}
\usepackage{amsmath}
\usepackage{amssymb}
\usepackage{graphics}
\usepackage{amsthm}
\usepackage[T1]{fontenc}
\usepackage{lmodern}
\usepackage[pdftex]{graphicx}
\usepackage{latexsym}
\usepackage{float}

\usepackage[T1]{fontenc}
\usepackage[english]{babel}	
\usepackage[a4paper]{typearea}
\usepackage{parskip}
\usepackage{mathtools}  
\usepackage{color}      

\DeclareGraphicsExtensions{.jpg, .eps}

\newtheorem{thm}{Theorem}
\newtheorem{prop}{Proposition}
\newtheorem{defi}{Definition}
\newtheorem{rem}{Remark}

\newtheorem{lem}{Lemma}

\DeclareMathOperator{\Sym}{Sym}

\begin{document}
\title{A Rigidity Result for a Reduced Model of a  Cubic-to-Orthorhombic Phase Transition in the Geometrically Linear Theory of Elasticity\footnote{This work was part of the diploma thesis of the author written under the supervision of Prof. Felix Otto. She is greatly indebted to him for introducing her to the topic, insightful discussions, comments and advice. Further, she would like to thank her PhD advisor Prof. Herbert Koch for his constant support and advice. She thanks the Deutsche Telekomstiftung and the Hausdorff Center for Mathematics for financial support and the Max-Planck-Institute for Mathematics in the Sciences in Leipzig for its kind hospitality. }}
\author{Angkana R\"uland\footnote{Mathematisches Institut, Universit\"at Bonn, Endenicher Allee 60, 53115 Bonn, Germany}}
\maketitle

\begin{abstract}
We study a simplified two-dimensional model for a cubic-to-orthorhombic phase transition occuring in certain shape-memory-alloys. 
In the low temperature regime the linear theory of elasticity predicts various possible patterns of martensite arrangements:
Apart from the well known laminates this phase transition 
displays additional structures involving four martensitic variants -- so called crossing twins. \\
Introducing a variational model including surface energy, we show that these structures are rigid under small energy perturbations. Combined with an upper bound construction this
gives the optimal scaling behavior of incompatible microstructures. These results are related to papers by Capella and Otto, \cite{CO08}, \cite{CO10}, as well as to a paper by Dolzmann and M\"uller, \cite{DM}.
\end{abstract}

\tableofcontents

\section{Introduction}
\subsection{The Stress-free Setting}
Working in the framework of linear elasticity, six stress-free strains characterize the body-centered to face-centered cubic-to-orthorhombic phase transition:
\begin{align*}
& e^{(1)} = \epsilon \begin{pmatrix}
                     \phantom{-}1 & \phantom{-}\delta & \phantom{-}0\\
		    \phantom{-}\delta & \phantom{-}1 &\phantom{-}0\\
		    \phantom{-}0 & \phantom{-}0 & -2 
                    \end{pmatrix}, \;
 e^{(2)} = \epsilon \begin{pmatrix}
                     \phantom{-}1 & -\delta & \phantom{-}0\\
		    -\delta & \phantom{-}1 &\phantom{-}0\\
		    \phantom{-}0 & \phantom{-}0 & -2 
                    \end{pmatrix}, \\
&e^{(3)} = \epsilon \begin{pmatrix}
                    \phantom{-}1 & \phantom{-}0 & \phantom{-}\delta\\
		    \phantom{-}0 & -2 & \phantom{-}0\\
		    \phantom{-}\delta & \phantom{-}0 & \phantom{-}1
                   \end{pmatrix}, \;
e^{(4)} = \epsilon \begin{pmatrix}
                    \phantom{-}1 & \phantom{-}0 & -\delta\\		
		    \phantom{-}0 & -2 & \phantom{-}0\\
		    -\delta & \phantom{-}0 & \phantom{-}1
                   \end{pmatrix}, \\
&e^{(5)} = \epsilon \begin{pmatrix}
                    -2 & \phantom{-}0 & \phantom{-}0\\
		    \phantom{-}0 & \phantom{-}1 & \phantom{-}\delta \\
		    \phantom{-}0 & \phantom{-}\delta & \phantom{-}1
                   \end{pmatrix}, \;
e^{(6)} = \epsilon \begin{pmatrix}
                    -2 & \phantom{-}0 & \phantom{-}0\\
		    \phantom{-}0 & \phantom{-}1 & -\delta \\
		    \phantom{-}0 & -\delta & \phantom{-}1
                   \end{pmatrix}.
\end{align*}
Here $\epsilon$ and $\delta$ are dimensionless parameters of typical magnitude $\sim 0.01$ and $\sim 0.25$, respectively.
Stress-free configurations of phases are therefore solutions to the 6-well problem:
\begin{align*}
 e(u) \in \{ e^{(1)}, ..., e^{(6)} \},
\end{align*}
where $e(u)= \frac{\nabla u + (\nabla u)^{t}}{2}$ is the strain tensor describing the relative changes of length and $u$ is a displacement field indicating how much a particle has
been moved under the deformation.\\

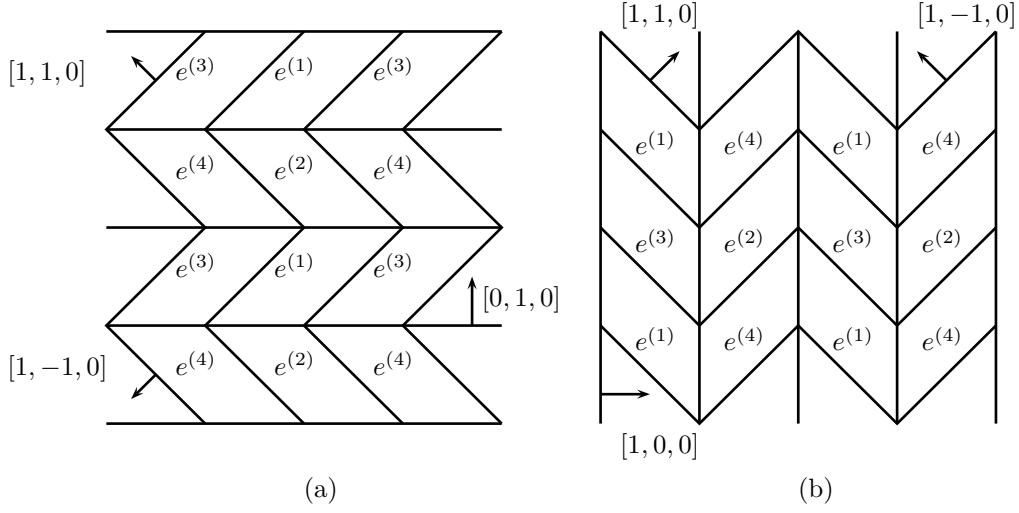
\begin{figure}[h]
\begin{pspicture}(-7.8,-1.25)(4,6)
\psset{unit=1.3cm}
\psset{linewidth=1pt}
\psline(-5,0)(-1,0)
\psline(-5,1)(-1,1)
\psline(-5,2)(-1,2)
\psline(-5,3)(-1,3)
\psline(-5,4)(-1,4)

\psline(-4,0)(-5,1)(-4,2)(-5,3)(-4,4)
\psline(-3,0)(-4,1)(-3,2)(-4,3)(-3,4)
\psline(-2,0)(-3,1)(-2,2)(-3,3)(-2,4)
\psline(-1,0)(-2,1)(-1,2)(-2,3)(-1,4)

\put(-4.3,0.5){$e^{(4)}$}
\put(-3.3,0.5){$e^{(2)}$}
\put(-2.3,0.5){$e^{(4)}$}

\put(-4.3,1.5){$e^{(3)}$}
\put(-3.3,1.5){$e^{(1)}$}
\put(-2.3,1.5){$e^{(3)}$}

\put(-4.3,2.5){$e^{(4)}$}
\put(-3.3,2.5){$e^{(2)}$}
\put(-2.3,2.5){$e^{(4)}$}

\put(-4.3,3.5){$e^{(3)}$}
\put(-3.3,3.5){$e^{(1)}$}
\put(-2.3,3.5){$e^{(3)}$}

\psline{->}(-1.3,1)(-1.3,1.5)
\put(-1.2,1.2){$[0,1,0]$}

\psline{->}(-4.5,0.5)(-4.75,0.25)
\put(-6,0.5){$[1,-1,0]$}

\psline{->}(-4.5,3.5)(-4.75,3.75)
\put(-6,3.5){$[1,1,0]$}

\put(-3,-0.75){(a)}

\psset{linewidth=1pt}
\psline(0,0)(0,4)
\psline(1,0)(1,4)
\psline(2,0)(2,4)
\psline(3,0)(3,4)
\psline(4,0)(4,4)
\psline(0,1)(1,0)(2,1)(3,0)(4,1)
\psline(0,2)(1,1)(2,2)(3,1)(4,2)
\psline(0,3)(1,2)(2,3)(3,2)(4,3)
\psline(0,4)(1,3)(2,4)(3,3)(4,4)
\psline{->}(0,0.3)(0.5,0.3)
\put(0.2,-0.3){$[1,0,0]$}
\psline{->}(0.5,3.5)(0.8,3.8)
\put(0.2,4.1){$[1,1,0]$}
\psline{->}(3.5,3.5)(3.2,3.8)
\put(3.2,4.1){$[1,-1,0]$}
\put(.35,2.75){$e^{(1)}$}
\put(.35,1.75){$e^{(3)}$}
\put(.35,.75){$e^{(1)}$}
\put(1.25,2.75){$e^{(4)}$}
\put(1.25,1.75){$e^{(2)}$}
\put(1.25,.75){$e^{(4)}$}
\put(2.35,2.75){$e^{(1)}$}
\put(2.35,1.75){$e^{(3)}$}
\put(2.35,.75){$e^{(1)}$}
\put(3.25,2.75){$e^{(4)}$}
\put(3.25,1.75){$e^{(2)}$}
\put(3.25,.75){$e^{(4)}$}
\put(2,-0.75){(b)}
\end{pspicture}
\caption{Possible crossing twin structures in the $y_{1},y_{2}$-plane: (a) corresponds to $e_{12}=e_{12}(y_{2})$, (b) corresponds to $e_{12}=e_{12}(y_{1})$.
All possible crossing twin configurations of the simplified setting display a characteristic structure: They consist of double-laminates which are made of an ``outer structure'' -- here given
by pairs of variants 1 and 3, as well as 2 and 4  -- and an ``inner structure'' determining the precise arrangement of the phases -- here these are given by the twinning modes of the pairs
1 and 3 and 2 and 4, respectively.
One notices that the relative volume fractions $\theta_{i}$ are determined by these structures: 
In (a) we have $\frac{\theta_{1}}{\theta_{1}+\theta_{3}}= \frac{\theta_{2}}{\theta_{2}+\theta_{4}}$, whereas
in (b) the situation is described by $\frac{\theta_{1}}{\theta_{1}+\theta_{3}}= \frac{\theta_{4}}{\theta_{2}+\theta_{4}}$.}
\label{fig:1}
\end{figure}

In contrast to the cubic-to-tetragonal phase transition for which Dolzmann and M\"uller, \cite{DM}, proved that (locally) only simple laminates occur, experiments suggest 
that in the cubic-to-orthorhombic phase transition crossing twin structures, i.e. structures involving zig-zag-bands of four martensitic phases (c.f. Figure \ref{fig:1}), have to be expected.
In order to capture these configurations we consider a simplified model: We assume that the strains are two-dimensional and only four variants of martensite, say variant one to four, are present. 
Carrying out calculations in the piecewise affine setting, we observe that there are exactly two twinning connections between any of the martensites and that there exist precisely six
martensitic crossing twin structures involving planar four-fold corners. These four-fold corners can be iterated to form the crossing twin structures.\\
In this setting it proves to be advantageous to carry out a change of coordinates and to renormalize the strains:
Setting $$C= \frac{1}{\sqrt{2}}
\begin{pmatrix}
0 & 1 & 1 \\
\sqrt{2} &  0 & 0 \\
0 & 1 & -1 
\end{pmatrix} \cdot \sqrt{6}\delta
\begin{pmatrix}
\frac{1}{\sqrt{3}} & 0 & 0\\
0& \frac{\sqrt{3}}{\sqrt{2}\delta} & 0\\
0 & 0 & \frac{1}{\sqrt{3}} \end{pmatrix}, $$ we are left with the following matrices
\begin{align*}
&\tilde{e}^{(1)} = \frac{\epsilon}{2d} 
\begin{pmatrix}
\phantom{-}d_{1} & \phantom{-}1 & \phantom{-}1\\
\phantom{-}1 & \phantom{-}d_{2} & \phantom{-}1 \\
\phantom{-}1 & \phantom{-}1& \phantom{-}d_{3} 
\end{pmatrix},\;\;
\tilde{e}^{(2)} = \frac{\epsilon}{2d}
\begin{pmatrix}
\phantom{-}d_{1} & -1 & \phantom{-}1\\
-1 & \phantom{-}d_{2} & -1 \\
\phantom{-}1 & -1& \phantom{-}d_{3} 
\end{pmatrix}, \\
&\tilde{e}^{(3)} = \frac{\epsilon}{2d}
\begin{pmatrix}
\phantom{-}d_{1} & \phantom{-}1 & -1\\
\phantom{-}1 & \phantom{-}d_{2} & -1 \\
-1 & -1& \phantom{-}d_{3} 
\end{pmatrix}, \;\;
\tilde{e}^{(4)} = \frac{\epsilon}{2d}
\begin{pmatrix}
\phantom{-}d_{1} & -1 & -1\\
-1 & \phantom{-}d_{2} & \phantom{-}1 \\
-1 & \phantom{-}1& \phantom{-}d_{3} 
\end{pmatrix},
\end{align*}
where $d^{-1}= 6\delta^{2},\; d_{1}= -\frac{1}{3} ,\; d_{2}= \frac{3}{2\delta^{2}} ,\; d_{3}= -\frac{1}{3} $. In the sequel we will suppress the tildes in the notation.
In these coordinates the first four strain tensors are highly symmetric, therefore if we are interested in two dimensional strains depending on only two coordinate directions, we can w.l.o.g. consider configurations depending on the $y_{1}, \; y_{2}$-coordinates only. For these it is possible to show that generic configurations are given by crossing twin phase distributions. The main result of the stress-free setting
is formulated in the following theorem (c.f. \cite{R}):
\medskip

\begin{thm}
\label{thm:1}
Let $U\subset \mathbb{R}^{3}$ be open, convex.
Assume $e\in\Sym(3,\mathbb{R}), \; e= e(y_{1}, y_{2})= \frac{\nabla u + (\nabla u)^{t}}{2}$, $u\in W^{1,\infty}(U ,\mathbb{R}^{3})$ such that
$e(u)\in \left\{e^{(1)}, \; e^{(2)}, \; e^{(3)}, \; e^{(4)} \right\}$ in $U$.  
\begin{enumerate}
\item  Then the following dichotomy holds: 
\begin{equation*}
e_{12} = e_{12}(y_{1}) \; \mbox{ or } \; e_{12} = e_{12}(y_{2}).
\end{equation*} 
\item In case $e_{12} = e_{12}(y_{1})$ there exists a function $g(t)$ such that:
\begin{equation*}
\label{eq:thm1.1}
(e_{13}\circ\Phi)(s,t) = e_{12}(s)g(t) \mbox{ and } (e_{23}\circ\Phi)(s,t) = g(t),
\end{equation*} 
where $\Phi(s,t) = (s, -E_{12}(s) + t) $ and $E_{12}^{\prime}(y_{1}) = e_{12}(y_{1}), \;
E_{12}(0)=0.$ 
Due to symmetry, $e= e(y_{1},y_{2})$ can also be replaced by $e(y_{1}, y_{3})$ and $e(y_{2},y_{3})$ respectively which yields analogous results. 
\end{enumerate}
\end{thm}

In this model involving only four (two-dimensional) strains, we do not have to require additional BV-regularity for the strain tensors.\\
\medskip

\subsection{The Setting of Small Deviations from the Stress-free Situation}
In the present paper we are interested in the rigidity of these constructions. In the spirit of the papers \cite{CO08}, \cite{CO10} of Capella and Otto, we introduce a variational model consisting of an elastic
and a surface energy contribution rescaled in an optimal manner. 
More precisely, we use an elastic energy of the form
\begin{equation*}
 E_{elast} = 2 \mu\int\limits_{B_{L}}{\left|e - \chi_{1}e^{(1)} - \chi_{2}e^{(2)} - \chi_{3}e^{(3)} - \chi_{4}e^{(4)} \right|^{2}dy},
\end{equation*} 
where $\chi_{k}$ are the characteristic functions of the martensitic phases, i.e. $\chi_{k}\in \{ 0, 1 \}$, 
$\mu$ is a material constant -- the first Lam\'e-constant -- of the dimension $\frac{J}{m^{3}}$,
and $B_{L}\subset \mathbb{R}^{3}$ denotes the sample.
As we are dealing with the analogue of the stress-free setting, we will assume that all quantities involved only depend on the $y_{1},y_{2}$ variables. The elastic energy can be rewritten in terms
of the modified characteristic functions $\tilde{\chi}_{k}$ which are defined as
\begin{align*}
&\tilde{\chi}_{1} = 1 - 2(\chi_{2} + \chi_{3}),\\
&\tilde{\chi}_{2} = 1 - 2(\chi_{3} + \chi_{4}),\\
&\tilde{\chi}_{3} = 1 - 2(\chi_{2} + \chi_{4}),\\
&\chi_{1} + \chi_{2} + \chi_{3} + \chi_{4} = 1, \;\; \chi_{k}\in \{0, 1\}, \;\; \tilde{\chi}_{k}\in\{-1,1\}.
\end{align*}
With these the elastic energy takes the form:
$$ E_{elast}(e) = 2  \mu\int\limits_{B_{L}}{\left|e - \frac{\epsilon}{2d}
\begin{pmatrix}
d_{1} & \tilde{\chi}_{3} & \tilde{\chi}_{2}\\
\tilde{\chi}_{3} & d_{2} &\tilde{\chi}_{1}\\
\tilde{\chi}_{2} & \tilde{\chi}_{1} & d_{3}\\
\end{pmatrix}
 \right|^{2}dy},\;  e=e(u)(y_{1}, y_{2}), $$
where $d, d_{1}, d_{2}, d_{3}$ are the constants obtained from the change of coordinates in the stress-free setting.\\
As it is well known that this energy is not weakly lower semicontinuous, we cannot hope to prove rigidity in this framework. Thus, a surface energy punishing high oscillations
is introduced:
\begin{align*}
E_{surf} = \kappa \int\limits_{B_{L}}{|\nabla \chi_{1}| + |\nabla \chi_{2}| + |\nabla \chi_{3}| + |\nabla \chi_{4}|dy}.
\end{align*}
Here $\kappa$ is a further material parameter of the units $\frac{J}{m^{2}}$.
Working with these two energy contributions it is known that there are two regimes for incompatible microstructures. In \cite{KM}, Kohn and M\"uller point out that the choice of the 
regime depends on a single non-dimensional quantity:
$\eta:= \frac{2d^{2}\kappa}{\epsilon^{2}\mu L}$. On the one hand, if $\eta\gg 1$, low-energy incompatible microstructures are characterized by fine-scale oscillations of twins and an energy contribution
scaling as $ E \sim (\kappa L^{2})^{\frac{1}{2}}(2\frac{\epsilon^{2}}{4d^{2}}\mu L^{3})^{\frac{1}{2}} $. If on the other hand $\eta \ll 1$, branching is energetically preferred and yields a scaling behavior of 
$E \sim \left(\kappa L^{2} \right)^{\frac{2}{3}}\left(2\frac{\epsilon^{2}}{4d^{2}}L^{3}\mu\right)^{\frac{1}{3}}$.
Rescaling all quantities by their natural units and energy so as to capture the regime of $\eta \ll 1$, we have:
\begin{align*}
& y = L \hat{y}, \\
& e = \frac{\epsilon}{2d} \hat{e},\\
& E_{elast} = \frac{\epsilon^{2}L^{3}\mu}{2d^{2}}\hat{E}_{elast},\\
& E_{surf} = \kappa L^{2} \hat{E}_{surf},\\
& E = \left( \frac{2 d^{2}\kappa}{ \epsilon^{2} \mu L}\right)^{\frac{2}{3}} \frac{\epsilon^{2}}{2 d^{2}}L^{3}\mu\hat{E}.
\end{align*}
With this we can point out the relevant quantities in their non-dimensional versions:

\medskip

\begin{defi}
\label{def:11}
We are interested in
\begin{itemize}
	\item the strains $e = \frac{\nabla u + (\nabla u)^{t}}{2}$, where 
	$u:\mathbb{R}^{3}\rightarrow \mathbb{R}^{3}$, $u\in W^{1,\infty}(\mathbb{R}^{3}, \mathbb{R}^{3})$ and
	$e^{(i)}= e^{(i)}(y_{1},y_{2})$,
	\begin{align*}
	&{e}^{(1)} = \frac{\epsilon}{2d} 
	\begin{pmatrix}
	\phantom{-}d_{1} & \phantom{-}1 & \phantom{-}1\\
	\phantom{-}1 & \phantom{-}d_{2} & \phantom{-}1 \\
	\phantom{-}1 & \phantom{-}1& \phantom{-}d_{3} 
	\end{pmatrix},\;\;
	{e}^{(2)} = \frac{\epsilon}{2d}
	\begin{pmatrix}
	\phantom{-}d_{1} & -1 & \phantom{-}1\\
	-1 & \phantom{-}d_{2} & -1 \\
	\phantom{-}1 & -1& \phantom{-}d_{3} 
	\end{pmatrix}, \\
	&{e}^{(3)} = \frac{\epsilon}{2d}
	\begin{pmatrix}
	\phantom{-}d_{1} & \phantom{-}1 & -1\\
	\phantom{-}1 & \phantom{-}d_{2} & -1 \\
	-1 & -1& \phantom{-}d_{3} 
	\end{pmatrix}, \;\;
	{e}^{(4)} = \frac{\epsilon}{2d}
	\begin{pmatrix}
	\phantom{-}d_{1} & -1 & -1\\
	-1 & \phantom{-}d_{2} & \phantom{-}1 \\
	-1 & \phantom{-}1& \phantom{-}d_{3} 
	\end{pmatrix}, \;\;
	\end{align*}
	where $d^{-1}= 6\delta^{2},\; d_{1}= -\frac{1}{3} ,\; d_{2}= \frac{3}{2\delta^{2}} ,\; d_{3}= -\frac{1}{3} $, 
	\item the characteristic functions $\chi_{i}(y_{1}, y_{2}) \in  \{0,1\}$, $\; i\in \{1,\cdots, 4\},$ with
	$\chi_{1} + \chi_{2} + \chi_{3} + \chi_{4} = 1 $
	for the martensite phases,
	\item the modified characteristic functions
	\begin{align*}
	&\tilde{\chi}_{1} = 1 - 2(\chi_{2} + \chi_{3}),\\
	&\tilde{\chi}_{2} = 1 - 2(\chi_{3} + \chi_{4}),\\
	&\tilde{\chi}_{3} = 1 - 2(\chi_{2} + \chi_{4}),\\
	&\tilde{\chi}_{i}\in\{-1,1\}.
	\end{align*}
	These can be reformulated in terms of the original characteristic functions:
	\begin{align*}
	\chi_{1}= \frac{1}{4}(1 + \tilde{\chi}_{2} + \tilde{\chi}_{3} + \tilde{\chi}_{1}),\\
	\chi_{2}= \frac{1}{4}(1 + \tilde{\chi}_{2} - \tilde{\chi}_{3} - \tilde{\chi}_{1}),\\
	\chi_{3}= \frac{1}{4}(1 - \tilde{\chi}_{2} + \tilde{\chi}_{3} - \tilde{\chi}_{1}),\\
	\chi_{4}= \frac{1}{4}(1 - \tilde{\chi}_{2} - \tilde{\chi}_{3} + \tilde{\chi}_{1}),
	\end{align*}
	\item the parameter $\eta:= \frac{2d^{2}\kappa}{\epsilon^{2}\mu L}$,
	\item the elastic energy
	\begin{align*}
	E_{elast} = \int\limits_{B_{1}}{\left|
	e - 
	\begin{pmatrix}
	d_{1} & \tilde{\chi}_{3} & \tilde{\chi}_{2}\\
	\tilde{\chi}_{3} & d_{2} &\tilde{\chi}_{1}\\
	\tilde{\chi}_{2} & \tilde{\chi}_{1} & d_{3}
	\end{pmatrix}
	\right|^{2}dy}, \; e= e(y_{1}, y_{2}),
	\end{align*}
	\item the surface energy
	\begin{equation*}
	{E}_{surf} = \int\limits_{B_{1}}{(|\nabla \chi_{1}| + |\nabla \chi_{2}| + |\nabla \chi_{3}| + |\nabla \chi_{4}|)dy},
	\end{equation*}
	\item the total energy
	\begin{equation*}
	 E:= \eta^{\frac{1}{3}}E_{surf} + \eta^{-\frac{2}{3}}E_{elast}.
	\end{equation*}
	We will always suppose $E\leq1$ in the sequel. If we want to stress the dependence on $\eta$, $\chi$ or $e$ we will also use the notation
	$E_{\eta}(e,\chi)$.
\end{itemize}
\end{defi}

\medskip

We remark that one should think of $\tilde{\chi}_{1}, \tilde{\chi}_{2}, \tilde{\chi}_{3}$ as corresponding to $e_{23}, e_{13}, e_{12}$.
Mimicking the proof of the stress-free case, c.f. \cite{R}, we show that the configurations are rigid -- at least in a weak sense:

\medskip

\begin{thm}
\label{thm:2}
\begin{enumerate}
\item
Let $\tilde{\chi}_{i}, \; E$ be as in Definition \ref{def:11}, let $\eta \leq 1$. Then there exists $r>0$ and functions $f_{(100)}, \; f_{(010)}$, such that: 
\begin{eqnarray*}
\int\limits_{B_{r}}{|\tilde{\chi}_{3} - f_{(100)}|dx} \lesssim  E^{\frac{1}{2}} \;
\mbox{ or } \; \int\limits_{B_{r}}{|\tilde{\chi}_{3} - f_{(010)}|dx} \lesssim  E^{\frac{1}{2}}, 
\end{eqnarray*}
where $f_{(100)}= f(y_{1})\in\{-1,1\},$ $f_{(010)}= f(y_{2})\in \{-1,1\}$.
\item
Assume $\tilde{\chi}_{1}, \tilde{\chi}_{2}, \tilde{\chi}_{3}$ to be as in Definition \ref{def:11} but now 
$\tilde{\chi}_{i}: \mathbb{T}^{2} \rightarrow \{-1,1\}$, $\mathbb{T}^{2}:= \mathbb{R}^{2}/ \mathbb{Z}^{2}$ and suppose that $e$ one-periodic in any coordinate direction and
\begin{align*}
&{E}_{elast} = \int\limits_{\mathbb{T}^{2}}{\left|e - 
\begin{pmatrix}
d_{1} & \tilde{\chi}_{3} & \tilde{\chi}_{2}\\
\tilde{\chi}_{3} & d_{2} &\tilde{\chi}_{1}\\
\tilde{\chi}_{2} & \tilde{\chi}_{1} & d_{3}\\
\end{pmatrix}
 \right|^{2}dx}, \, e= e(y_{1}, y_{2}).
\end{align*} 
Further assume that $\theta_{i}:= \langle \chi_{i} \rangle := \int\limits_{\mathbb{T}^{2}}\chi_{i}dx $ and that the first case in (1) holds (i.e. $\tilde{\chi}_{3} \sim f_{(100)}$). \\
Then
\begin{align*}
 |\theta_{1}(\theta_{2} + \theta_{4}) - \theta_{4}(\theta_{1} + \theta_{3})| \lesssim E^{\frac{1}{4}}.
\end{align*}
\item Let the same assumptions as in (2) be satisfied. Suppose $\Phi(s,t)=(s,-F_{(100)}(s)+t)$ with $F_{(100)}^{\prime}(s) = f_{(100)}(s)$ a.e.,
$F_{(100)}(0)=0.$ \\
Then there exists $g: \Phi^{-1}\left(\left[-\frac{1}{2},\frac{1}{2}\right]^{2}\right) \rightarrow \mathbb{R}$, $(s,t)\mapsto g(t)$, such that
\begin{align*}
 \left\| \tilde{\chi}_{1}\circ\Phi - g \right\|^{2}_{L^{2}(\Phi^{-1}([-\frac{1}{2},\frac{1}{2}]^{2}))} \lesssim \eta^{-\frac{2}{9}}E^{\frac{5}{6}}
\end{align*}
and 
\begin{align*}
 \left\| \tilde{\chi}_{2}\circ\Phi - (f_{(100)}\circ\Phi)g \right\|_{L^{2}(\Phi^{-1}([-\frac{1}{2},\frac{1}{2}]^{2}))} \lesssim  \max\left\{ \eta^{-\frac{1}{9}}E^{\frac{5}{12}}, E^{\frac{1}{4}} \right\}.
\end{align*}
\end{enumerate}
Due to symmetry, similar results hold in the case $e=e(y_{1}, y_{3})$ or $e=e(y_{2},y_{3})$.
\end{thm}

\medskip

In this context the notation $A \lesssim B$ is used to denote that there exists a universal constant $C>0$ such that
$A \leq C B$.

As rigidity estimates always imply a lower bound on the scaling of incompatible microstructures (i.e. structures which do not display the right distribution of volume fractions),
and as the rigidity estimate can be complemented by an upper bound construction for incompatible microstructures, it is possible to prove the optimal scaling of incompatible microstructures. We obtain
\medskip

\begin{prop}
\label{cor:1}
Let $\eta \ll 1$. Then there exits a family of strains $e_{\eta}$ with
\begin{align*}
& |\theta^{\eta}_{1}(\theta^{\eta}_{2} + \theta^{\eta}_{4}) - \theta^{\eta}_{4}(\theta^{\eta}_{1}+\theta^{\eta}_{3})| \geq \frac{1}{5}, \\
& |\theta^{\eta}_{1}(\theta^{\eta}_{2} + \theta^{\eta}_{4}) - \theta^{\eta}_{2}(\theta^{\eta}_{1}+\theta^{\eta}_{3})| \geq \frac{1}{5}.
\end{align*} 
such that
\begin{align*}
  E_{\eta}(e_{\eta},\chi^{\eta}) \leq c
\end{align*}
for $c>0$ independent of $\eta$.
In particular, for such strains it holds
\begin{align*}
 \frac{1}{c} \leq \min
  E_{\eta}(e_{\eta},\chi^{\eta}) \leq c
\end{align*}
for $c>0$ independent of $\eta$.
\end{prop}

\section{The Proofs}

\subsection{Rigidity of the Outer Structure}

In this section we prove the (strong) rigidity of the outer structure of the patterns. We proceed in four major steps: Firstly, we apply the compatibility conditions for strains
to obtain weak control of $\tilde{\chi}_{3}$ (c.f. Lemma \ref{lem:OS-2}). To be more precise, we prove control of certain second order derivatives of $\tilde{\chi}_{3}$ in $H^{-2}$. Although these
expressions display the right scaling, the existence of incompatible microstructures suggests that the weak control cannot immediately be converted into strong $L^{2}$-control. 
For that purpose we therefore follow the strategy paved in \cite{CO10}: In a second step we pass to finite differences so as to obtain weak $H^{-1}$-control of these (c.f. Lemma \ref{lem:OS-1}).
In the decisive step we then interpolate between $BV$ and $H^{-1}$ to obtain $L^{2}$-control of $\tilde{\chi}_{3}$ (c.f. Lemma \ref{lem:interpol}). Finally, we conclude as in the stress-free setting (c.f. \cite{R}) and use a wave argument
and the two-valuedness of $\tilde{\chi}_{3}$ (c.f. Proposition \ref{prop:2}).

\medskip

We begin with the weak $H^{-2}$-control:

\medskip

\begin{lem}
\label{lem:OS-2}
 There exist functions $\rho_{11}, \; \rho_{12}, \; \rho_{22}:B_{1}(0) \rightarrow \mathbb{R}$ such that
\begin{equation}
 \label{eq:compcon}
 \partial_{1}\partial_{2}\tilde{\chi}_{3} = \partial_{11}\rho_{11} + \partial_{1}\partial_{2}\rho_{12} + \partial_{22}\rho_{22} \; \mbox{ in } \mathcal{D}^{\prime} 
\end{equation}
and
\begin{equation}
 \int\limits_{B_{1}}{\rho_{11}^{2} + \rho_{12}^{2} + \rho_{22}^{2}dx} \lesssim E_{elast}.
\end{equation}
\end{lem}

\medskip

\begin{proof}[Proof of Lemma \ref{lem:OS-2}]
As $e$ is a strain the compatibility conditions must be satisfied, in particular:
\begin{equation*}
\partial^{2}_{11}e_{22} - 2\partial_{1}\partial_{2}e_{12} + \partial^{2}_{22}e_{11} = 0.
\end{equation*}
This can be exploited to derive the identity
\begin{equation*}
2\partial_{1}\partial_{2}\tilde{\chi}_{3} = \partial_{11}(e_{22} - d_{2}) + 2\partial_{1}\partial_{2}( \tilde{\chi}_{3} - e_{12}) + \partial_{22}(e_{11} - d_{1}),
\end{equation*}
where $d_{1}, \; d_{2}$ are the (constant) diagonal entries of the tensors describing the cubic-to-orthorhombic phase transition.
Setting
\begin{eqnarray*}
\rho_{11} &:=& \frac{ e_{22} - d_{2}}{2},\\
\rho_{12} &:=& \tilde{\chi}_{3} - e_{12},  \\
\rho_{22} &:=& \frac{e_{11} - d_{1}}{2},
\end{eqnarray*}
and noticing that these functions are some of the components of the elastic energy the claim follows.
\end{proof}

\medskip

In order to apply the interpolation inequality we have to argue via finite differences. Hence, for $v\in\mathbb{R}^{n}$, $h\in\mathbb{R}$ we set
\begin{align*}
 \partial_{v}^{h}f(x) := f(x+ hv) - f(x).
\end{align*}

\medskip

\begin{lem}
\label{lem:OS-1}
There exist functions $j_{11}, \; j_{22},\; j: B_{1}(0) \rightarrow \mathbb{R}$ such that
\begin{align*}
& \partial_{1}^{h_{1}}\partial_{2}^{h_{2}}\tilde{\chi}_{3} = \partial_{1}j_{11} + \partial_{2}j_{22} + j \mbox{ in } \mathcal{D}^{\prime},\\
& \int\limits_{B_{1}(0)}{j_{11}^{2} + j_{22}^{2} + j^{2}dx} \lesssim E_{elast}.
\end{align*}
\end{lem}

\medskip

\begin{proof}[Proof of Lemma \ref{lem:OS-1}]
The claim of the lemma follows from an integration of the identity derived in Lemma \ref{lem:OS-2}:
We have
\begin{align*}
\partial_{1}\partial_{2}\tilde{\chi}_{3} &= \partial_{11}\rho_{11} + \partial_{1}\partial_{2}\rho_{12} + \partial_{22}\rho_{22}\\
\Leftrightarrow \partial_{1}\partial_{2}(\tilde{\chi}_{3} - \rho_{12}) &= \partial_{11}{\rho_{11}} + \partial_{22}\rho_{22}.
\end{align*}
Evaluating this expression at $y=x+h_{1}e_{1} + h_{2}e_{2}$ and reinterpreting the derivatives, we obtain 
\begin{align*}
& \frac{d}{dh_{1}}\frac{d}{dh_{2}}(\tilde{\chi}_{3} - \rho_{12})(x + h_{1}e_{1} + h_{2}e_{2})\\
&= \frac{d}{dh_{1}}\partial_{1}\rho_{11}(x + h_{1}e_{1} + h_{2}e_{2})
+ \frac{d}{dh_{2}}\partial_{2}\rho_{22}(x + h_{1}e_{1} + h_{2}e_{2}).
\end{align*}
This can be integrated to arrive at 
\begin{align*}
\MoveEqLeft{\partial_{1}^{h_{1}}\partial_{2}^{h_{2}}(\tilde{\chi}_{3} - \rho_{12})(x)
= \int\limits_{0}^{h_{2}}\int\limits_{0}^{h_{1}}{\frac{d}{dh^{\prime}_{1}}\frac{d}{dh^{\prime}_{2}}(\tilde{\chi}_{3} - \rho_{12})(x + h^{\prime}_{1}e_{1} + h^{\prime}_{2}e_{2})dh^{\prime}_{1}dh^{\prime}_{2}}}\\
= & \; \int\limits_{0}^{h_{2}}\int\limits_{0}^{h_{1}}{\frac{d}{dh^{\prime}_{1}}\partial_{1}\rho_{11}(x+h^{\prime}_{1}e_{1} + h^{\prime}_{2}e_{2})dh^{\prime}_{1}dh^{\prime}_{2}}\\
& + \int\limits_{0}^{h_{2}}\int\limits^{h_{1}}_{0}{\frac{d}{dh^{\prime}_{2}}\partial_{2}\rho_{22}(x+h^{\prime}_{1}e_{1} +h^{\prime}_{2}e_{2})dh^{\prime}_{1}dh^{\prime}_{2}}\\
= & \; \partial_{1}\int\limits_{0}^{h_{2}}{\rho_{11}(x + h_{1}e_{1} + h^{\prime}_{2}e_{2})dh^{\prime}_{2}}  - \partial_{1}\int\limits_{0}^{h_{2}}{\rho_{11}(x + h^{\prime}_{2}e_{2})dh^{\prime}_{2}}\\
& + \partial_{2}\int\limits_{0}^{h_{1}}{\rho_{22}(x + h^{\prime}_{1}e_{1} + h_{2}e_{2})dh^{\prime}_{1}} - \partial_{2}\int\limits_{0}^{h_{1}}{\rho_{22}(x + h_{1}^{\prime}e_{1})dh_{1}^{\prime}}.  
\end{align*}
Hence, the statement of the lemma follows with the functions
\begin{align*}
j_{11} &:= \partial_{1}\int\limits_{0}^{h_{2}}{\rho_{11}(x + h_{1}e_{1} + h_{2}^{\prime}e_{2})dh^{\prime}_{2}} 
- \partial_{1}\int\limits_{0}^{h_{2}}{\rho_{11}(x + h^{\prime}_{2}e_{2})dh^{\prime}_{2}},\\
j_{22} &:= \partial_{2}\int\limits_{0}^{h_{1}}{\rho_{22}(x + h^{\prime}_{1}e_{1} + h_{2}e_{2})dh^{\prime}_{1}}
- \partial_{2}\int\limits_{0}^{h_{1}}{\rho_{22}(x + h_{1}^{\prime}e_{1})dh_{1}^{\prime}},\\
j &:= \partial_{1}^{h_{1}}\partial_{2}^{h_{2}}\rho_{12}.
\end{align*} 
\end{proof}

\medskip

In the following lemma the transition from weak to strong norms is achieved via a well-known interpolation estimate, c.f. \cite{CO10}.

\medskip

\begin{lem}
\label{lem:interpol}
 Let $\eta\leq 1$. There exists a universal radius $r>0$ such that
\begin{equation*}
\sup\limits_{|h_{1}|, |h_{2}| \leq r}\int\limits_{B_{r}}{|\partial^{h_{1}}_{1}\partial^{h_{2}}_{2}\tilde{\chi}_{3}|dx} \lesssim \eta^{-\frac{2}{3}}E_{elast} + \eta^{\frac{1}{3}}E_{surf}.
\end{equation*}
\end{lem}

\medskip
\begin{proof}
We reason via the estimate
\begin{equation*}
\int\limits_{B_{r}}{f^{2}dx} \lesssim   \eta^{\frac{1}{3}}\int\limits_{B_{1}}{|\nabla f|dx}\sup|f| + \eta^{-\frac{2}{3}}\int\limits_{B_{1}}||\nabla|^{-1}f|^{2}dx,
\end{equation*}
applied to $\partial_{1}^{h_{1}}\partial_{2}^{h_{2}}\tilde{\chi}_{3}$. 
Using the discreteness of the values of $\tilde{\chi}_{3}$, we notice that the $L^{2}$-norm is equivalent to the $L^{1}$-norm and that the $L^{\infty}$-norm is bounded by a uniform constant.
Thus, we conclude
\begin{equation*}
\int\limits_{B_{r}}{|\partial_{1}^{h_{1}}\partial_{2}^{h_{2}}\tilde{\chi}_{3}| dx} \lesssim   \eta^{\frac{1}{3}}\int\limits_{B_{1}}{|\nabla \partial_{1}^{h_{1}}\partial_{2}^{h_{2}}\tilde{\chi}_{3} |dx}
 + \eta^{-\frac{2}{3}}\int\limits_{B_{1}}(|j|^{2}+ j_{0}^{2})dx, 
\end{equation*}
where $j= \begin{pmatrix}
           j_{1}\\ j_{2}
          \end{pmatrix}
$ and $j_{0}$ are given by Lemma \ref{lem:OS-1}.
\end{proof}

\medskip

Having established strong control, we can mimic the wave argument from the stress-free case. To simplify notation we use the following convention:

\medskip

\begin{defi}
Let $\{a,b\}\subset \mathbb{R}^{2}$ be a basis with dual basis given by $\{a^{\ast}, b^{\ast}\}$. 
For $f:\mathbb{R} \rightarrow \mathbb{R}$ define the following notation $f_{a^{\ast}}(x):= f(a^{\ast}\cdot x)$.
\end{defi}

\medskip

\begin{lem}
\label{lem:discrwave}
Let $\{a,b\}\subset \mathbb{R}^{2}$ be a basis. For all functions $f:B_{2}(0)\subset \mathbb{R}^{2} \rightarrow \mathbb{R}$ there exist $r>0$ and functions
$g_{a^{\ast}}, \; g_{b^{\ast}}$ such that
\begin{equation*}
\int\limits_{B_{r}}{|f - g_{a^{\ast}} - g_{b^{\ast}}|dx} \leq C(a,b,r) \sup\limits_{|h_{a}|, |h_{b}| \leq 1}\int\limits_{B_{1}}{|\partial_{a}^{h_{a}}\partial_{b}^{h_{b}}f|dx}.
\end{equation*}.
\end{lem}

\medskip

\begin{proof}[Proof of Lemma \ref{lem:discrwave}]
As $a,\;b$ form a basis, we can without loss of generality assume $a=e_{1}, \; b=e_{2}$. Furthermore, we can replace the balls by cubes. This yields the following estimate
\begin{align*}
\MoveEqLeft {\sup\limits_{|h_{1}|, |h_{2}| \leq 1}\int\limits_{(-1,1)^{2}}{|\partial_{1}^{h_{1}}\partial_{2}^{h_{2}}f|dx}}\\
& =   \sup\limits_{|h_{1}|, |h_{2}| \leq 1}\int\limits_{(-1,1)^{2}}{|f(x_{1} + h_{1}, x_{2} + h_{2}) - f(x_{1} +h_{1}, 	x_{2}) - f(x_{1}, x_{2} + h_{2}) +f(x_{1}, x_{2})|dx}\\
& \geq \frac{1}{4} \int\limits_{(-1,1)^{2}}\int\limits_{(-1,1)^{2}}{|f(x_{1} + h_{1}, x_{2} + h_{2}) -f(x_{1}, x_{2}+ h_{2})  - f(x_{1} + h_{1}, x_{2}) + f(x_{1}, x_{2})|dxdh}\\
& \geq  \frac{1}{4} \int\limits_{(-\frac{1}{2} , 																																			\frac{1}{2})^{2}}\int\limits_{(-\frac{1}{2},\frac{1}{2})^{2}}{|f(y_{1}, y_{2}) 	-f(x_{1}, y_{2})  - f(y_{1}, x_{2}) 	+ f(x_{1}, x_{2})|dxdy}\\	
& \stackrel{(\ast)}{\geq} \frac{1}{4} \int\limits_{(- \frac{1}{2}, \frac{1}{2})^{2}}{\left|\int\limits_{(-\frac{1}{2}, 									\frac{1}{2})^{2}}	f(x_{1}, x_{2}) - f(x_{1}, y_{2}) - f(y_{1}, x_{2}) + f(y_{1}, y_{2})dy \right|dx}\\
& = \frac{1}{4} \int\limits_{(- \frac{1}{2}, \frac{1}{2})^{2}}{\left| f(x_{1},x_{2}) - 	\int\limits_{-\frac{1}{2}}^{\frac{1}{2}}{f(x_{1},y_{2})dy_{2}} - \int\limits_{-\frac{1}{2}}^{\frac{1}{2}}{f(y_{1}, x_{2})dy_{1}} - 
	\int\limits_{\frac{1}{2}}^{\frac{1}{2}}\int\limits_{\frac{1}{2}}^{\frac{1}{2}}{f(y_{1}, y_{2} )dy_{1} dy_{2}}  \right| dx},
\end{align*}
where $(\ast)$ is a consequence of Jensen's inequality.\\
Thus, the statement holds with the functions
\begin{align*}
g_{(100)}(x_{1}) &= \int\limits_{-\frac{1}{2}}^{\frac{1}{2}}{f(x_{1},y_{2})dy_{2}} + 
		\frac{1}{2} \int\limits_{\frac{1}{2}}^{\frac{1}{2}}\int\limits_{\frac{1}{2}}^{\frac{1}{2}}{f(y_{1}, y_{2} )dy_{1} dy_{2}}, \\
g_{(010)}(x_{2}) &= \int\limits_{-\frac{1}{2}}^{\frac{1}{2}}{f(y_{1},x_{2})dy_{1}} + 
		\frac{1}{2} \int\limits_{\frac{1}{2}}^{\frac{1}{2}}\int\limits_{\frac{1}{2}}^{\frac{1}{2}}{f(y_{1}, y_{2} )dy_{1} dy_{2}}. \qedhere	
\end{align*} 
\end{proof}

\medskip

\begin{prop}
 \label{prop:2}
Let $\tilde{\chi}_{3}, \;E$ be as in Definition \ref{def:11}. Then there exist a universal radius $r>0$ and functions $f_{(100)}, \; f_{(010)}\in\{-1,1\}$ with the following properties:
\begin{eqnarray*}
\int\limits_{B_{r}}{|\tilde{\chi}_{3} - f_{(100)}|dx} \lesssim  E^{\frac{1}{2}} \;
\vee \; \int\limits_{B_{r}}{|\tilde{\chi}_{3} - f_{(010)}|dx} \lesssim  E^{\frac{1}{2}}.
\end{eqnarray*} 
\end{prop}
 
\medskip

\begin{proof}[Proof of Proposition \ref{prop:2}]
We divide the proof into several steps:\\
\textbf{Step 1}: \itshape Application of Lemma \ref{lem:interpol} and Lemma \ref{lem:discrwave}.\\
\upshape As by assumption $\tilde{\chi}_{3}$ only depends on two variables, Lemma \ref{lem:discrwave} can be applied in combination with 
Lemma \ref{lem:interpol}. This yields the existence of a radius $r>0$ and of functions $g_{(100)}, \; g_{(010)}: B_{r}\subset \mathbb{R}^{2} \rightarrow \mathbb{R}$ such that
\begin{align*}
 \int\limits_{B_{r}}|\tilde{\chi}_{3} -g_{(100)} -g_{(010)}|dx \lesssim E.
\end{align*}
\textbf{Step 2}: \itshape There exist functions $\tilde{g}_{(100)}, \; \tilde{g}_{(010)}: B_{1}(0) \subset \mathbb{R}^{2} \rightarrow \mathbb{R}$ such that 
\begin{align}
\label{eq:2.1}	
 & \tilde{g}_{(100)} \in \left\{a_{1}-1, a_{1} + 1 \right\}, \;
 \tilde{g}_{(010)} \in \left\{a_{2}-1, a_{2} + 1 \right\}, \; a_{1},a_{2}\in \mathbb{R}, \\
\label{eq:2.2} 
 &\int\limits_{B_{r}}{| \tilde{g}_{(100)} -g_{(100)} |dx} \lesssim E, \;
 \int\limits_{B_{r}}{| \tilde{g}_{(010)} -g_{(010)} |dx} \lesssim E.
\end{align} 
\upshape It suffices to prove the statement for $\tilde{g}_{(100)}$. Since for any $L^{\infty}$-function, $f$, we find $x^{\ast}_{2}\in(-1,1)$ such that the evaluation of $f$
at $x_{2}^{\ast}$ is less than the mean value of $f$ in $x_{2}\in (-1,1)$, we conclude
\begin{align*}
& \exists  \; x_{2}^{\ast}\in (-r,r): \\
& \int\limits_{(-r,r)}{|\tilde{\chi}_{3}(x_{1}, x_{2}^{\ast}) - g_{(100)}(x_{1}) - g_{(010)}(x_{2}^{\ast})|dx_{1}}\\
& \leq \frac{1}{2r}\int\limits_{(-r,r)^{2}}{|\tilde{\chi}_{3}(x_{1},x_{2}) -g_{(100)}(x_{1}) -g_{(010)}(x_{2})|dx}.
\end{align*}
Setting
\begin{align*}
\tilde{g}_{(100)}(x_{1}) := \tilde{\chi}_{3}(x_{1}, x_{2}^{\ast}) - g_{(010)}(x_{2}^{\ast}),
\end{align*}
(\ref{eq:2.1}) and (\ref{eq:2.2}) follow immediately.\\ 
An application of the triangle inequality and of step 2 implies that it suffices to show\\
\textbf{Step 3:} \itshape For $\tilde{g}_{(100)}, \; \tilde{g}_{(010)}$ we find $a\in\mathbb{R}$ such that
\begin{align*}
& \min\left\{ \int\limits_{B_{r}}| \tilde{g} _{(100)} -a|dx, \int\limits_{B_{r}}{| \tilde{g} _{(010)} -a|dx}  \right\} \\
& \lesssim \left( \int\limits_{B_{r}}{|\tilde{\chi}_{3} - \tilde{g}_{(100)} - \tilde{g}_{(010)}|dx} \right)^{\frac{1}{2}}
\lesssim E^{\frac{1}{2}}.
\end{align*}
\upshape Let
\begin{align*}
\lambda_{1}&:= \mathcal{L}^{1}\left(\left\{ x_{1}\in (-r,r); \; \tilde{g}_{(100)}(x_{1}) = a_{1}-1 \right\}\right),\\
\lambda_{2}&:= \mathcal{L}^{1}\left(\left\{ x_{2}\in (-r,r); \; \tilde{g}_{(010)}(x_{2}) = a_{2}-1 \right\}\right),\\
\epsilon &:=\int\limits_{(-r,r)^{2}}{|\tilde{\chi}_{3} - \tilde{g}_{(100)} -\tilde{g}_{(010)} |dx}.
\end{align*}
We can estimate
\begin{align*}
\epsilon 
\geq & \; \lambda_{1}\lambda_{2}\mbox{dist}(a_{1}+a_{2}-2, \left\{ -1, 1 \right\})\\
& \; + (2r-\lambda_{1})(2r-\lambda_{2})\mbox{dist}(a_{1} + a_{2} +2, \left\{-1,1\right\}).
\end{align*}
Moreover, we must have that either $\mbox{dist}(a_{1}+a_{2}-2, \left\{ -1, 1 \right\}) \geq \frac{1}{2}$ or \\
$\mbox{dist}(a_{1} + a_{2} +2, \left\{-1,1\right\}) \geq \frac{1}{2} $. Otherwise, the inequality $\mbox{dist}(a_{1}+a_{2}-2, \left\{ -1, 1 \right\})< \frac{1}{2}$
would imply that either
\begin{align*}
a_{1} + a_{2} \in \left(\frac{1}{2}, \frac{3}{2} \right) \mbox{ or }
a_{1} + a_{2} \in \left(\frac{5}{2}, \frac{7}{2} \right).
\end{align*}
holds. In the first case, however, this would yield
\begin{equation*}
a_{1} + a_{2} +2 \in \left(\frac{5}{2}, \frac{7}{2} \right).
\end{equation*}
In the second case, this would result in
\begin{equation*}
a_{1} + a_{2} +2 \in \left(\frac{9}{2}, \frac{11}{2} \right).
\end{equation*}
Both statements contradict the assumption $\mbox{dist}(a_{1} + a_{2} +2, \left\{-1,1\right\}) < \frac{1}{2} $.
Therefore, we can w.l.o.g. assume $\mbox{dist}(a_{1}+a_{2}-2, \left\{-1,1\right\}) \geq \frac{1}{2}$. This results in:
\begin{align*}
& 2\epsilon \geq \lambda_{1} \lambda_{2}\\
& \Rightarrow \lambda_{1} \leq \sqrt{2}\sqrt{\epsilon} \; \vee \; \lambda_{2} \leq \sqrt{2}\sqrt{\epsilon}\\
& \Rightarrow \int\limits^{r}_{-r}{|\tilde{g}_{(100)} -(a_{1}+1)|dx_{1}} = 2\lambda_{1} \leq 2\sqrt{2 \epsilon} \lesssim E^{\frac{1}{2}}\\
& \vee \int\limits^{r}_{-r}{|\tilde{g}_{(010)} -(a_{2}+1)|dx_{2}} = 2\lambda_{2} \leq 2\sqrt{2 \epsilon} \lesssim E^{\frac{1}{2}}.
\end{align*}
An analogous argument works in case $\mbox{dist}(a_{1} + a_{2} +2, \left\{-1,1\right\}) \geq \frac{1}{2}$.\\
\textbf{Step 4}: \itshape Conclusion.\\
\upshape Without loss of generality, we may suppose that the second alternative of step 2 holds, i.e. there exists $g_{(100)}: (-r,r)\rightarrow \mathbb{R}$  such that
\begin{align*}
 \int\limits_{(-r,r)^{2}}{|\tilde{\chi}_{3}(x_{1},x_{2}) - g_{(100)}(x_{1})|dx_{1}dx_{2}} \lesssim E^{\frac{1}{2}}.
\end{align*}
Since $|g_{(100)}| \lesssim 1$ this leads to 
\begin{align*}
 \int\limits_{(-r,r)^{2}}{|\tilde{\chi}_{3}(x_{1},x_{2}) - g_{(100)}(x_{1})|^{2}dx_{1}dx_{2}} \lesssim E^{\frac{1}{2}}.
\end{align*}
As the $L^{2}$-projection on the space of constants is given by the mean value of the respective function, this implies
\begin{align*}
 \int\limits_{(-r,r)^{2}}{|\tilde{\chi}_{3}(x_{1},x_{2}) - \frac{1}{2r}\int\limits_{-r}^{r}\tilde{\chi}_{3}(x_{1},x_{2}^{\prime})dx_{2}^{\prime} |^{2}dx_{1}dx_{2}} \lesssim E^{\frac{1}{2}}.
\end{align*}
Defining
\begin{align*}
 \tilde{\chi}_{3}^{\ast}(x_{1}) :=
\left\{
\begin{array}{ll}
 1; & \frac{1}{2r}\int\limits_{-r}^{r}{\tilde{\chi}_{3}(x_{1},x_{2})dx_{2}} \geq 0,\\
 -1; & \mbox{else},
\end{array} \right.
\end{align*}
and remarking $|\tilde{\chi}_{3}(x_{1},x_{2}) - \frac{1}{2r}\int\limits_{-r}^{r}{\tilde{\chi}_{3}(x_{1},x_{2})dx_{2}} | \geq 1$ on $\left\{ \tilde{\chi}_{3} \neq \tilde{\chi}_{3}^{\ast} \right\}$, we note
\begin{align*}
\MoveEqLeft {E^{\frac{1}{2}} \gtrsim \int\limits_{\{ \tilde{\chi}_{3} \neq \tilde{\chi}_{3}^{\ast} \}\cap (-r,r)^{2}}|\tilde{\chi}_{3}(x_{1},x_{2}) - \frac{1}{2r}\int\limits_{-r}^{r}{\tilde{\chi}_{3}(x_{1},x_{2}^{\prime})dx_{2}^{\prime}}|^{2}dx_{1}dx_{2}}\\
& \geq \mathcal{L}^{2}(\{ \tilde{\chi}_{3} \neq \tilde{\chi}_{3}^{\ast} \}).
\end{align*}
Hence, we obtain
\begin{align*}
\MoveEqLeft{\int\limits_{(-r,r)^{2}}{|\tilde{\chi}_{3} - \tilde{\chi}_{3}^{\ast}|^{2}dx_{1}dx_{2}} 
= \int\limits_{\{ \tilde{\chi}_{3} \neq \tilde{\chi}_{3}^{\ast} \}\cap(-r,r)^{2}}{|\tilde{\chi}_{3}(x_{1},x_{2}) - \tilde{\chi}_{3}^{\ast}(x_{1})|^{2}dx_{1}dx_{2}}}\\
& \leq 4 \mathcal{L}^{2}(\left\{  \tilde{\chi}_{3} \neq \tilde{\chi}_{3}^{\ast} \right\} \cap (-r,r)^{2}) \lesssim E^{\frac{1}{2}}.
\end{align*}
Setting $f_{(100)}:= \tilde{\chi}_{3}^{\ast}$ and noticing that the $L^{1}$-estimate follows from the $L^{2}$-estimate as a consequence of the discreteness of $\tilde{\chi}_{3} - \tilde{\chi}_{3}^{\ast}$,
the claim follows.
\end{proof}

\medskip

\subsection{Proof of the Second Result}

In this section we consider the inner structure of the patterns. To avoid technical difficulties we work in a periodic setting. However, as the argument for the inner structure
is of local nature in the stress-free case, \cite{R}, we believe that this condition can be removed and replaced with a purely local reasoning.\\
Again, we follow the ideas of the stress-free setting (c.f. \cite{R}): Using the compatibility conditions, we begin with proving weak control (Lemma \ref{lem:H-2per}). Via Helmholtz decomposition, we obtain that
$\begin{pmatrix}
  \tilde{\chi}_{2} \\ \tilde{\chi}_{1}
 \end{pmatrix}$ is close to a gradient field (Lemma \ref{lem:grad}). Using the method of characteristics, we argue that this gradient field is $H^{-1}$-close to a function of a single variable (Proposition \ref{prop:3}). 
Last but not least, this can be translated into a statement on the volume fractions of the modified characteristic functions (Proposition \ref{prop:4}).

\medskip

We recall the setting:

\begin{defi}
\label{defi:FT}
Let $ \mathbb{T}^{d} := \mathbb{R}^{d}/  \mathbb{Z}^{d} $ and $f: \mathbb{T}^{d} \rightarrow \mathbb{R}, \; f\in L^{1}(\mathbb{T}^{d})$. Define
\begin{align*}
& \mathcal{F}f(k) := \int\limits_{\mathbb{T}^{d}}{f(x)e^{-2\pi i k\cdot x}dx},  \; k\in \mathbb{Z}^{d},\\
& \int\limits_{\mathbb{Z}^{d}}{h(k)dk} := \sum\limits_{k\in  \mathbb{Z}^{d}}{h(k)}.
\end{align*}
Let $f: \mathbb{T}^{d}\rightarrow \mathbb{R}$ be measurable. Set
\begin{align*}
& \left\| f \right\|_{H^{-s}(\mathbb{T}^{d})}^{2} := \int\limits_{\mathbb{Z}^{d}}{\frac{|\mathcal{F}f|^{2}}{|k|^{2s}}dk},\\
& f\in H^{-s}(\mathbb{T}^{d}) \Leftrightarrow \left\| f \right\|_{H^{-s}(\mathbb{T}^{d})}^{2}<\infty,\\
& \left\| f \right\|^{2}_{H^{-1}_{full}(\mathbb{T}^{d})}:= \int\limits_{\mathbb{Z}^{d}}{\frac{1}{1+|k|^{2}}|\mathcal{F}f|^{2}dk},\\
& f\in H^{-1}_{full}(\mathbb{T}^{d}) \Leftrightarrow \left\| f \right\|_{H^{-1}_{full}(\mathbb{T}^{d})}^{2}<\infty.
\end{align*}
Let $f: \mathbb{R}^{d} \rightarrow \mathbb{R}$ be measurable, $M \subset \mathbb{R}^{d}$ Borel, we define
\begin{align*}
 \langle f \rangle_{M} := \frac{1}{\mathcal{L}^{d}(M)} \int\limits_{M}{f(y)dy}.
\end{align*}
In the sequel we use the convention
\begin{align*}
& \chi_{1}, \chi_{2}, \chi_{3}, \chi_{4}: \mathbb{T}^{2}\rightarrow \mathbb{R},\\
& \tilde{\chi}_{1}, \tilde{\chi}_{2}, \tilde{\chi}_{3}: \mathbb{T}^{2} \rightarrow \mathbb{R},\\
& {E}_{elast} =  \int\limits_{[-\frac{1}{2},\frac{1}{2}]^{2}}{\left|e - 
\begin{pmatrix}
d_{1} & \tilde{\chi}_{3} & \tilde{\chi}_{2}\\
\tilde{\chi}_{3} & d_{2} &\tilde{\chi}_{1}\\
\tilde{\chi}_{2} & \tilde{\chi}_{1} & d_{3}\\
\end{pmatrix}
 \right|^{2}dy}, \\
& e= e(y_{1}, y_{2}), \; e: \mathbb{T}^{2} \rightarrow \Sym(3,\mathbb{R}).\\
\end{align*} 
\end{defi}

\medskip

\begin{lem}
\label{lem:H-2per}
For configurations in the $y_{1},y_{2}$-plane ($e = e(y_{1}, y_{2})$) we have
\begin{align*}
\partial_{1}(\partial_{1}\tilde{\chi}_{1} - \partial_{2}\tilde{\chi}_{2}) & = \partial_{1}\partial_{1}\phi_{11} + \partial_{1}\partial_{2}\phi_{12},\\
\partial_{2}(\partial_{1}\tilde{\chi}_{1} - \partial_{2}\tilde{\chi}_{2}) & = \partial_{1}\partial_{2}\rho_{12} + \partial_{2}\partial_{2}\rho_{22} 
\end{align*}
and
\begin{align*}
\int\limits_{[-\frac{1}{2}, \frac{1}{2}]^{2}}{\rho_{12}^{2} + \rho_{22}^{2} dx
} \lesssim E_{elast},\\
\int\limits_{[-\frac{1}{2},\frac{1}{2}]^{2}}{\phi_{11}^{2} + \phi_{12}^{2} dx
} \lesssim E_{elast}.
\end{align*}
In other words, for the $\left[ -\frac{1}{2}, \frac{1}{2} \right]^{2}$-periodic characteristic functions $\tilde{\chi}_{2}, \; \tilde{\chi}_{1}$, this is equivalent to $H^{-2}$-control: 
\begin{align*}
\left\|\nabla \left( \nabla \times \begin{pmatrix} \tilde{\chi}_{2} \\ \tilde{\chi}_{1} \end{pmatrix} \right) \right\|^{2}_{H^{-2}(\mathbb{T}^{2})} \lesssim  E_{elast}.
\end{align*}
\end{lem}

\medskip

\begin{proof}[Proof of Lemma \ref{lem:H-2per}]
As in Lemma \ref{lem:OS-2} the statement is a result of the compatibility conditions for strains. We make use of the second block of the equations yielding 
\begin{align*}
0 &= \partial_{1}(-\partial_{1}e_{23} + \partial_{2}e_{13}),\\
0 &= \partial_{2}(\partial_{1}e_{23} - \partial_{2}e_{13}),
\end{align*}
for configurations in the plane spanned by $y_{1}, \; y_{2}$.
Thus, we obtain
\begin{align*}
\partial_{2}(\partial_{1}\tilde{\chi}_{1} - \partial_{2}\tilde{\chi}_{2}) 
= & \; \partial_{2}(\partial_{1}(\tilde{\chi}_{1} - e_{23}) - \partial_{2}(\tilde{\chi}_{2} - e_{13}) 
)\\
=: & \; \partial_{1}\partial_{2}\rho_{12} - \partial_{22}{\rho_{22}}. 
\end{align*}
Again noticing that the $\rho_{ij}$ correspond to the components of the elastic energy, we obtain $L^{2}$-control:
\begin{align*}
\int\limits_{[-\frac{1}{2},\frac{1}{2}]^{2}}{\rho_{12}^{2} + \rho_{22}^{2} dx} \lesssim E_{elast}.
\end{align*}
The second statement follows in the same way.
\end{proof}

\medskip

\begin{lem}
\label{lem:grad}
Let $w: \mathbb{T}^{2} \rightarrow \mathbb{R}^{2}$, $w\in L^{2}(\mathbb{T}^{2})$,  then we have
\begin{equation}
\label{eq:grad1}
\left\|Pw\right\|_{L^{2}(\mathbb{T}^{2})} = \left\|\nabla \times w\right\|_{H^{-1}(\mathbb{T}^{2})},
\end{equation}
where $P$ denotes the Leray-projection.
\end{lem}

\medskip

\begin{proof}[Proof of Lemma \ref{lem:grad}]
Working in Fourier space, the Leray-projection takes the following form:
\begin{equation*}
\mathcal{F}(Pw) = \mathcal{F}w - \frac{k \cdot \mathcal{F}w}{|k|^{2}}k = \frac{|k|^{2}\mathcal{F}w - k \cdot \mathcal{F}w}{|k|^{2}}.
\end{equation*}
With the identity
\begin{equation*}
 k\times (k \times \mathcal{F}w) = |k|^{2}\mathcal{F}w - k\cdot \mathcal{F}w,
\end{equation*}
we obtain
\begin{align*}
 |\mathcal{F}(Pw)|^{2} & = \frac{|k\times( k \times \mathcal{F}w)|^{2}}{|k|^{4}}\\
		      & = \frac{|k\times \mathcal{F}w|^{2}}{|k|^{2}},
\end{align*}
which proves the claim.
\end{proof}

\medskip

\begin{prop}
\label{prop:3}
Let $\tilde{\chi}_{1}, \tilde{\chi}_{2}, \tilde{\chi}_{3}, \; E_{elast}$ be as in Definition \ref{defi:FT}. 
Define $\Phi(s,t) := (s,t -F_{(100)}(s))$ where $F_{(100)}^{\prime}(s) = f_{(100)}(s)$ a.e., $F_{(100)}(0)= 0$ and where $f_{(100)}$ is the function from Proposition \ref{prop:2}.
Then we have:
\begin{enumerate}
\item 
There exists $u:[-\frac{1}{2},\frac{1}{2}]^{2}\rightarrow \mathbb{R}$, $u\in H^{1}([-\frac{1}{2},\frac{1}{2}]^{2})$, $\left[-\frac{1}{2}, \frac{1}{2}\right]^{2}$ - periodic, and there exists	
$g: \Phi^{-1}([-\frac{1}{2}, \frac{1}{2}]^{2})  \subset \mathbb{R}^{2} \rightarrow \mathbb{R}$, $(s,t) \mapsto g(t)$, one-periodic, such that
\begin{align*}
&\left\|
\begin{pmatrix}
\tilde{\chi}_{2}\\
\tilde{\chi}_{1}\\
\end{pmatrix}
- \nabla u
\right\|_{L^{2}([-\frac{1}{2}, \frac{1}{2}]^{2})}
\lesssim E^{\frac{1}{2}}_{elast},\\
& \left\|
u\circ\Phi - g
\right\|_{L^{2}(\Phi^{-1}([-\frac{1}{2}, \frac{1}{2}]^{2}))} 
\lesssim E^{\frac{1}{4}}.
\end{align*} 
 \item Let $\tilde{g}: \mathbb{R}^{2} \rightarrow \mathbb{R}$, $(y_{1},y_{2}) \mapsto \tilde{g}(y_{1},y_{2})$ be such that
 $\tilde{g}(y_{1},y_{2}) = (g\circ\Phi^{-1})(y_{1},y_{2})$ for $(y_{1},y_{2}) \in [-\frac{1}{2}, \frac{1}{2}]^{2}$ and let $\tilde{g}$ be $[-\frac{1}{2}, \frac{1}{2}]^{2}$-periodically continued.
Then we have
\begin{align*}
	 &\left\|
\begin{pmatrix}
\tilde{\chi}_{2}\\
\tilde{\chi}_{1}\\
\end{pmatrix}
-
\begin{pmatrix}
  f_{(100)}\partial_{2}\tilde{g}\\
 \partial_{2}\tilde{g}	
\end{pmatrix}
\right\|_{H^{-1}_{full}([-\frac{1}{2}, \frac{1}{2}]^{2})} \lesssim E^{\frac{1}{4}}.
\end{align*}
\end{enumerate} 
\end{prop}

\begin{rem}
\label{rem:7}
 \begin{enumerate}
  \item Let $f: \mathbb{R}^{2} \rightarrow \mathbb{R}$, $(s,t) \mapsto f(s,t)$.\\
    As $\det(D\Phi)(s,t) = 1 $, we have
    \begin{align*}
      \left\| f \right\|_{L^{2}(\Phi^{-1}([-\frac{1}{2}, \frac{1}{2}]^{2}))} = \left\| f\circ \Phi^{-1} \right\|_{L^{2}([-\frac{1}{2}, \frac{1}{2}]^{2})}.
    \end{align*}
  \item In the periodic setting the statement of Proposition \ref{prop:3} amounts to 
    \begin{align*}
    \left\| \tilde{\chi}_{3} - f_{(100)} \right\|_{L^{2}([-\frac{1}{2}, \frac{1}{2}]^{2})} \lesssim E^{\frac{1}{4}}.
    \end{align*}
\end{enumerate}
\end{rem}

\medskip
Mimicking the proof of the stress-free case, we proceed in several steps: Via the compatibility conditions we prove closeness of $\begin{pmatrix}
                                                                                                                                \tilde{\chi}_{2} \\ \tilde{\chi}_{1}
                                                                                                                               \end{pmatrix}$ 
to a gradient field $\nabla u$ for which we 
determine the characteristic equations. An application of Poincar\'e's inequality and a change of coordinates yield closeness of $u$ to a function of a single variable.
Since $u$ resembles, loosely speaking, the inverse gradient of  $\begin{pmatrix}
                                                   \tilde{\chi}_{2}\\ \tilde{\chi}_{1}
                                                  \end{pmatrix}$,
this motivates the closeness of this vector field to the crossing twin structures with respect to the $H^{-1}$-norm. Without additionally making use of the surface energy this is optimal (c.f. counterexample given in Lemma \ref{lem:counter}).

\medskip

\begin{proof}
\textbf{Step 1:} \itshape $H^{-1}$-control.\\
\upshape
Lemma \ref{lem:H-2per} yields
\begin{align*}
\left\| \nabla \left(\nabla \times
\begin{pmatrix}
\tilde{\chi}_{2} \\
\tilde{\chi}_{1}
\end{pmatrix}
 \right) \right\|_{H^{-2}([-\frac{1}{2}, \frac{1}{2}]^{2})}\lesssim E_{elast}^{\frac{1}{2}}.
\end{align*} 
Using the periodicity assumptions this immediately translates into  
\begin{align*}
\left\| \nabla \times
\begin{pmatrix}
\tilde{\chi}_{2} \\
\tilde{\chi}_{1}
\end{pmatrix}
\right\|_{H^{-1}([-\frac{1}{2}, \frac{1}{2}]^{2})}
=   
\left\| \nabla \left(\nabla \times
\begin{pmatrix}
\tilde{\chi}_{2} \\
\tilde{\chi}_{1}
\end{pmatrix}
 \right) \right\|_{H^{-2}([-\frac{1}{2}, \frac{1}{2}]^{2})}
 \lesssim E_{elast}^{\frac{1}{2}},
\end{align*}
which is the estimate we looked for.

\textbf{Step 2: }\itshape $L^{2}$-control.\\
\upshape
Denoting the Leray projection with $P$ and referring to Lemma \ref{lem:grad}, we obtain
\begin{align*}
\left\|
P\left(
\begin{pmatrix}
\tilde{\chi}_{2}\\
\tilde{\chi}_{1}
\end{pmatrix}\right)
 \right\|_{L^{2}([-\frac{1}{2}, \frac{1}{2}]^{2})}
\stackrel{(\ref{eq:grad1})}{=} \left\|\nabla \times
\begin{pmatrix}
\tilde{\chi}_{2}\\
\tilde{\chi}_{1}
\end{pmatrix}
 \right\|_{H^{-1}([-\frac{1}{2}, \frac{1}{2}]^{2})} \lesssim E^{\frac{1}{2}}_{elast}. 
\end{align*}
With the Helmholtz-projection on gradient fields, $Q$, this turns into
\begin{align*}
\left\|
\begin{pmatrix}
\tilde{\chi}_{2}\\
\tilde{\chi}_{1}
\end{pmatrix}
- Q
\begin{pmatrix}
\tilde{\chi}_{2}\\
\tilde{\chi}_{1}
\end{pmatrix}
 \right\|_{L^{2}([-\frac{1}{2}, \frac{1}{2}]^{2})} \lesssim E_{elast}^{\frac{1}{2}}
\end{align*}
and 
\begin{align*}
&\begin{pmatrix}
\partial_{1}u\\
\partial_{2}u
\end{pmatrix}
:=
Q
\begin{pmatrix}
\tilde{\chi}_{2}\\
\tilde{\chi}_{1}
\end{pmatrix}  \in L^{2}\left(\left[-\frac{1}{2}, \frac{1}{2}\right]^{2}\right), \; u \in H^{1}\left(\left[-\frac{1}{2}, \frac{1}{2}\right]^{2}\right),\\
&u \; \left[-\frac{1}{2},\frac{1}{2}\right]^{2} \mbox{- periodic}.
\end{align*}

\textbf{Step 3: }\itshape Characteristics for $u$: We have
\begin{align}
 \label{eq:P3.1}
 \left\| \frac{d}{ds}(u\circ\Phi) \right\|_{L^{2}(\Phi^{-1}([-\frac{1}{2},\frac{1}{2}]^{2}))} \lesssim E^{\frac{1}{4}}.
\end{align}

\upshape As in the stress-free setting we exploit the structure of the strains describing the phase transition.
All in all, we have the following identities and estimates at our disposal:
\begin{align}
\label{eq:p2.1}
& \tilde{\chi}_{2} - \tilde{\chi}_{3}\tilde{\chi}_{1} = 0,\\
\label{eq:p2.2}
& \left\| \tilde{\chi}_{3} - f_{(100)} \right\|_{L^{2}([-\frac{1}{2}, \frac{1}{2}]^{2})} \lesssim E^{\frac{1}{4}},\\
\label{eq:p2.3}
& \left\|
	\begin{pmatrix}
	\tilde{\chi}_{2}\\
	\tilde{\chi}_{1}
	\end{pmatrix}
	-
	Q
	\begin{pmatrix}
	\tilde{\chi}_{2}\\
	\tilde{\chi}_{1}
	\end{pmatrix}
 \right\|_{L^{2}([-\frac{1}{2}, \frac{1}{2}]^{2})} \lesssim E_{elast}^{\frac{1}{2}},\\
\label{eq:p2.4}
& |f_{(100)}| = 1, \,\; \eta\leq 1.
\end{align}
Since $\Phi(s,t)$ is a bilipschitz mapping the chain rule may be applied almost everywhere:
\begin{align*} 
\MoveEqLeft{\left\|\frac{d}{ds}(u \circ\Phi) \right\|_{L^{2}(\Phi^{-1}([-\frac{1}{2},\frac{1}{2}]^{2}))}
= \left\|\partial_{1}u\circ\Phi - (f_{(100)}\circ\Phi)(\partial_{2}u\circ\Phi)\right\|_{L^{2}(\Phi^{-1}([-\frac{1}{2},\frac{1}{2}]^{2}))}}\\
& \stackrel{(\ref{eq:p2.1})}{=} \left\|\partial_{1}u - f_{(100)}\partial_{2}u  + f_{(100)}\tilde{\chi}_{1} -   f_{(100)}\tilde{\chi}_{1} - \tilde{\chi}_{2} + \tilde{\chi}_{3}\tilde{\chi}_{1} \right\|_{L^{2}([-\frac{1}{2},\frac{1}{2}]^{2})}\\
& \leq  \underbrace{\left\| \partial_{1}u - \tilde{\chi}_{2} \right\|_{L^{2}([-\frac{1}{2},\frac{1}{2}]^{2})}}_{\stackrel{(\ref{eq:p2.3})}{\lesssim} E_{elast}^{\frac{1}{2}}}
 + \underbrace{\left\|f_{(100)}(\partial_{2}u - \tilde{\chi}_{1}) \right\|_{L^{2}([-\frac{1}{2},\frac{1}{2}]^{2})}}_{\stackrel{(\ref{eq:p2.3}),(\ref{eq:p2.4})}{\lesssim} \left\|\partial_{2}u - \tilde{\chi}_{1} \right\|_{L^{2}([-\frac{1}{2},\frac{1}{2}]^{2})}}\\
& + \underbrace{\left\|\tilde{\chi}_{1}(f_{(100)} - \tilde{\chi}_{3}) \right\|_{L^{2}([-\frac{1}{2},\frac{1}{2}]^{2})}}_
	{\leq \left\|\tilde{\chi}_{1} \right\|_{L^{\infty}}\left\|f_{(100)} - \tilde{\chi}_{3} \right\|_{L^{2}([-\frac{1}{2},\frac{1}{2}]^{2})} }
 \stackrel{(\ref{eq:p2.3}),(\ref{eq:p2.2})}{\lesssim}  E^{\frac{1}{4}}.	
\end{align*}
\textbf{Step 4: }\itshape Application of Poincar\'e's inequality.\\
\upshape As $\Phi$ is bilipschitz, there exist $k\in\mathbb{N}$ and a cube $\left[-\frac{r}{2},\frac{r}{2} \right]^{2}$, such that the inclusions $\Phi^{-1}([-\frac{1}{2}, \frac{1}{2}]^{2}) \subset \left[-\frac{r}{2},\frac{r}{2} \right]^{2} \subset \Phi^{-1}([-\frac{k}{2}, \frac{k}{2}]^{2})$ hold.
Due to the $\left[-\frac{1}{2}, \frac{1}{2}\right]^{2}$ - periodicity of $u$ we have
\begin{align*}
 \left\| \frac{d}{ds}\left(u\circ\Phi\right)  \right\|_{L^{2}(\Phi^{-1}([-\frac{k}{2}, \frac{k}{2}]^{2}))} \leq Ck^{2}\left\| \frac{d}{ds}\left(u\circ\Phi\right)  \right\|_{L^{2}(\Phi^{-1}([-\frac{1}{2},\frac{1}{2}]^{2}))} \lesssim E^{\frac{1}{4}}.
\end{align*}
Applying the lemma of Poincar\'e we find $g: \left[-\frac{r}{2},\frac{r}{2} \right]^{2} \subset \mathbb{R}^{2} \rightarrow \mathbb{R}$, $(s,t) \mapsto g(t)$, one-periodic, such that
\begin{align*}
\left\|u\circ\Phi - g\right\|_{L^{2}\left(\left[-\frac{r}{2},\frac{r}{2} \right]^{2}\right)} \lesssim \left\|\frac{d}{ds}\left(u\circ\Phi\right) \right\|_{L^{2}\left(\left[-\frac{r}{2},\frac{r}{2} \right]^{2}\right)}.
\end{align*}
Therefore we conclude
\begin{align*}
\MoveEqLeft{ \left\| u\circ\Phi - g \right\|_{L^{2}(\Phi^{-1}\left([-\frac{1}{2}, \frac{1}{2}]^{2}\right))}
 \leq  \left\| u\circ\Phi - g \right\|_{L^{2}\left(\left[-\frac{r}{2},\frac{r}{2} \right]^{2}\right)}}\\
& \lesssim \left\| \frac{d}{ds}\left(u\circ\Phi\right) \right\|_{L^{2}\left(\left[-\frac{r}{2},\frac{r}{2} \right]^{2}\right)} 
\leq \left\| \frac{d}{ds}\left(u\circ\Phi\right) \right\|_{L^{2}(\Phi^{-1}([-\frac{k}{2}, \frac{k}{2}]^{2}))}\\
& \lesssim E^{\frac{1}{4}}.
\end{align*}

\textbf{Step 5: }\itshape Proof of part 2.\\
\upshape We estimate:
\begin{align*}
\MoveEqLeft {\left\|
 \begin{pmatrix}
  \tilde{\chi}_{2}\\
  \tilde{\chi}_{1} 			
 \end{pmatrix}
 - \begin{pmatrix}
    f_{(100)}\partial_{2}\tilde{g}\\
	\partial_{2}\tilde{g}
   \end{pmatrix}
 \right\|_{H^{-1}_{full}(\mathbb{T}^{2})}
\leq 
 \left\| \begin{pmatrix}
             \tilde{\chi}_{2}\\
	     \tilde{\chi}_{1}
            \end{pmatrix}
 -\nabla u
 \right\|_{H^{-1}_{full}(\mathbb{T}^{2})}}\\
& + 
\left\|  \begin{pmatrix}
 f_{(100)}\partial_{2}\tilde{g}\\
  \partial_{2}\tilde{g}			
 \end{pmatrix} - 
\begin{pmatrix}
  f_{(100)}\partial_{2}u\\
\partial_{2}u
\end{pmatrix}\right\|_{H^{-1}_{full}(\mathbb{T}^{2})}\\
& + \left\| \partial_{1}u -f_{(100)}\partial_{2}u
 \right\|_{H^{-1}_{full}(\mathbb{T}^{2})}.
\end{align*}
Using the continuous embedding $L^{2} \hookrightarrow H^{-1}_{full}$ and recalling (\ref{eq:p2.3}), 
we can deal with the first term:
\begin{align*}
  \left\| \begin{pmatrix}
             \tilde{\chi}_{2}\\
	     \tilde{\chi}_{1}
            \end{pmatrix}
 -\nabla u
 \right\|_{H^{-1}_{full}(\mathbb{T}^{2})}
\lesssim \left\|  \begin{pmatrix}
             \tilde{\chi}_{2}\\
	     \tilde{\chi}_{1}
            \end{pmatrix}
 -\nabla u\right\|_{L^{2}(\mathbb{T}^{2})} \lesssim E_{elast}^{\frac{1}{2}}.
\end{align*}
For the third term we recall that by definition of $\Phi$ the identity
\begin{align*}
 (\partial_{1}u - f_{(100)}\partial_{2}u)\circ\Phi = \frac{d}{ds}(u\circ\Phi)
\end{align*}
holds. Taking into account (\ref{eq:P3.1}) and Remark \ref{rem:7}, we obtain:
\begin{align*}
\MoveEqLeft{ \left\| \partial_{1}u -f_{(100)}\partial_{2}u
 \right\|_{H^{-1}_{full}(\mathbb{T}^{2})} \lesssim 
\left\| \partial_{1}u -f_{(100)}\partial_{2}u
 \right\|_{L^{2}(\mathbb{T}^{2})}}\\
& = \left\| \frac{d}{ds}(u\circ\Phi) \right\|_{L^{2}(\Phi^{-1}(\mathbb{T}^{2}))} 
\stackrel{(\ref{eq:P3.1})}{\lesssim} E^{\frac{1}{4}}.
\end{align*}

In order to bound the second term we use that $f_{(100)}$ only depends on $y_{1}$. 
Further we remark that $\partial_{2}\tilde{g}$ exists in the Sobolev sense as $g$ is one-periodic and $\tilde{g}$ is obtained from $g$ via periodizing in $y_{1}$-direction.
Consequently:
\begin{align*}
 \left\|  \begin{pmatrix}
 f_{(100)}\partial_{2}\tilde{g}\\
  \partial_{2}\tilde{g}			
 \end{pmatrix} - 
\begin{pmatrix}
  f_{(100)}\partial_{2}u\\
\partial_{2}u
\end{pmatrix}\right\|_{H^{-1}_{full}(\mathbb{T}^{2})} &= 
 \left\|  \partial_{2} \left( \begin{pmatrix}
 f_{(100)}\tilde{g}\\
  \tilde{g}			
 \end{pmatrix} - 
\begin{pmatrix}
  f_{(100)}u\\
u
\end{pmatrix} \right) \right\|_{H^{-1}_{full}(\mathbb{T}^{2})}\\
&\leq  \left\|  \begin{pmatrix}
 f_{(100)}\tilde{g}\\
  \tilde{g}			
 \end{pmatrix} - 
\begin{pmatrix}
  f_{(100)}u\\
u
\end{pmatrix} \right\|_{L^{2}(\mathbb{T}^{2})}.
\end{align*}
Due to (\ref{eq:p2.4}), we deduce 
\begin{align*}
\MoveEqLeft{ \left\|  \begin{pmatrix}
 f_{(100)}\tilde{g}\\
  \tilde{g}			
 \end{pmatrix} - 
\begin{pmatrix}
  f_{(100)}u\\
u
\end{pmatrix} \right\|_{L^{2}(\mathbb{T}^{2})} \leq 2 \left\| \tilde{g} - u\right\|_{L^{2}(\mathbb{T}^{2})}} \\
& \lesssim \left\| g\circ\Phi^{-1} - u \right\|_{L^{2}(\mathbb{T}^{2})}\\
& = \left\| g - u\circ\Phi \right\|_{L^{2}(\Phi^{-1}([-\frac{1}{2},\frac{1}{2}]^{2}))}
 \lesssim E^{\frac{1}{4}}.
\end{align*}
Thus, we combine these estimates to conclude
\begin{align*}
 \left\|
 \begin{pmatrix}
  \tilde{\chi}_{2}\\
  \tilde{\chi}_{1} 			
 \end{pmatrix}
 - \begin{pmatrix}
    f_{(100)}\partial_{2}\tilde{g}\\
	\partial_{2}\tilde{g}
   \end{pmatrix}
 \right\|_{H^{-1}_{full}(\mathbb{T}^{2})} \lesssim E^{\frac{1}{4}}.
\end{align*}

\end{proof}

\medskip

In order derive a statement on the volume fractions, we have to exploit the properties of the approximating functions associated to the modified characteristic functions. Under the change
of coordinates, $\Phi$, the functions $f_{(100)}$ and $\tilde{g}$ approximating $\tilde{\chi}_{3}$ and $\tilde{\chi}_{1}$ respectively are, roughly speaking, independent. The good properties
of the change of coordinates preserve this.

\medskip

\begin{prop}
\label{prop:4}
Let $\eta\leq 1$ and assume $\tilde{\chi}_{i}$, $E$, $E_{elast}$ are as in Definition \ref{defi:FT}, suppose that $\left\| \tilde{\chi}_{3} - f_{(100)} \right\|_{L^{2}(\mathbb{T}^{2})} \lesssim E^{\frac{1}{4}}$. Let $\Phi(s,t):= (s,t-F_{(100)}(s))$, $F^{\prime}_{(100)}(s)= f_{(100)}(s)$ a.e., $F_{(100)}(0)=0$. 
Then it holds
\begin{align*}
 \left| \theta_{1}(\theta_{2} + \theta_{4}) - \theta_{4}(\theta_{1} + \theta_{3}) \right| \lesssim E^{\frac{1}{4}}.
\end{align*}
\end{prop}

\medskip

\begin{proof}[Proof of Proposition \ref{prop:4}]
\textbf{Step 1:} \itshape Uncorrelatedness. Let $f, g: \mathbb{R}^{2} \rightarrow \mathbb{R} $; suppose $f, g$ to be one-periodic in $y_{2}$ and assume
\begin{align*}
 &(f\circ\Phi)(s,t) = h(s),\\
 &(g\circ\Phi)(s,t) = l(t). 
\end{align*}
Then we have (using the notation of Definition \ref{defi:FT})
\begin{align*}
 \langle f \rangle_{\mathbb{T}^{2}} \langle g \rangle _{\mathbb{T}^{2}}
= \langle f g \rangle_{\mathbb{T}^{2}}.
\end{align*}
\upshape Exploiting the properties of $\Phi$ and using the periodicity of $f$, $g$, we obtain: 
\begin{align*}
\int\limits_{\Phi([-\frac{1}{2},\frac{1}{2}]^{2})}{f(y_{1},y_{2})dy_{1}dy_{2}}
& = \int\limits_{-\frac{1}{2}}^{\frac{1}{2}}\int\limits_{-F_{(100)}(y_{1})-\frac{1}{2}}^{-F_{(100)}(y_{1})+\frac{1}{2}}{f(y_{1},y_{2})dy_{2}dy_{1}}\\
& = \int\limits_{-\frac{1}{2}}^{\frac{1}{2}}\int\limits_{-\frac{1}{2}}^{\frac{1}{2}}{f(y_{1},y_{2})dy_{1}dy_{2}}.
\end{align*}
Thus, we have
\begin{align*}
& \langle f \rangle _{\mathbb{T}^{2}} = \langle f \rangle_{\Phi([-\frac{1}{2},\frac{1}{2}]^{2})},\\
& \langle g \rangle _{\mathbb{T}^{2}} = \langle g\rangle_{\Phi([-\frac{1}{2},\frac{1}{2}]^{2})},\\
& \langle fg \rangle _{\mathbb{T}^{2}} = \langle fg \rangle_{\Phi([-\frac{1}{2},\frac{1}{2}]^{2})}.
\end{align*}
An application of Fubini's theorem finally proves the claim
\begin{align*}
 \langle f \rangle _{\mathbb{T}^{2}}\langle g \rangle _{\mathbb{T}^{2}}
&= \langle f \rangle _{\Phi([-\frac{1}{2},\frac{1}{2}]^{2})}\langle g \rangle _{\Phi([-\frac{1}{2},\frac{1}{2}]^{2})}
 = \langle f\circ\Phi \rangle_{[-\frac{1}{2},\frac{1}{2}]^{2}}\langle g\circ\Phi \rangle_{[-\frac{1}{2},\frac{1}{2}]^{2}} \\
&= \left( \int\limits_{-\frac{1}{2}}^{\frac{1}{2}} h(s) ds \right) \left( \int\limits_{-\frac{1}{2}}^{\frac{1}{2}} l(t) dt \right) 
= \int\limits_{-\frac{1}{2}}^{\frac{1}{2}}\int\limits_{-\frac{1}{2}}^{\frac{1}{2}}h(s)l(t)dsdt \\
& = \langle (fg)\circ\Phi \rangle_{[-\frac{1}{2},\frac{1}{2}]^{2}} = \langle fg \rangle_{\Phi([-\frac{1}{2},\frac{1}{2}]^{2})}
  = \langle fg \rangle_{\mathbb{T}^{2}}.
\end{align*}

\textbf{Step 2:} \itshape Proof of the proposition.\\
\upshape 
By definition of $\tilde{\chi}_{i}$ we can rewrite the expressions $\theta_{1}(\theta_{2} + \theta_{4})$,
$\theta_{4}(\theta_{1} + \theta_{3})$ as combinations of the modified characteristic functions:
\begin{align}
\label{eq:P5.1}
 & \theta_{1}(\theta_{2} + \theta_{4}) = \left \langle \frac{1-\tilde{\chi}_{3}}{2} \right \rangle \left\langle \frac{1}{4}(1 + \tilde{\chi}_{2} + \tilde{\chi}_{3} + \tilde{\chi}_{1}) \right\rangle, \\
\label{eq:P5.2}
 & \theta_{4}(\theta_{1} + \theta_{3}) = \left \langle 1- \frac{1-\tilde{\chi}_{3}}{2} \right \rangle \left\langle \frac{1}{4}(1 - \tilde{\chi}_{2} - \tilde{\chi}_{3} + \tilde{\chi}_{1}) \right\rangle. 
\end{align}
Due to the $H^{-1}_{full}$-estimate from Proposition \ref{prop:3}, the mean values of $\tilde{\chi}_{2}$, $\tilde{\chi}_{1}$ are controlled by functions describing the crossing twin structures:
\begin{align}
& \left( \sum\limits_{k\in\mathbb{Z}^{2}}{\frac{1}{1+|k|^{2}}\left| \begin{pmatrix}  (\mathcal{F}\tilde{\chi}_{2})(k)\\ (\mathcal{F}\tilde{\chi}_{1})(k) \end{pmatrix}  - \begin{pmatrix} (\mathcal{F}f_{(100)}\partial_{2}\tilde{g})(k)\\ \mathcal{F}(\partial_{2}\tilde{g})(k) \end{pmatrix}
 \right|^{2} } \right)^{\frac{1}{2}} \lesssim E^{\frac{1}{4}} \nonumber\\
& \Rightarrow \left| \begin{pmatrix}  (\mathcal{F}\tilde{\chi}_{2})(0)\\ (\mathcal{F}\tilde{\chi}_{1})(0) \end{pmatrix} - \begin{pmatrix} (\mathcal{F}f_{(100)}\partial_{2}\tilde{g})(0)\\ (\mathcal{F}\partial_{2}\tilde{g})(0) \end{pmatrix} \right| \lesssim E^{\frac{1}{4}}\nonumber \\
\label{eq:P5.3}
& \Leftrightarrow
 \left\{ \begin{array}{ll}
 \left| \langle \tilde{\chi}_{2} \rangle  -  \langle f_{(100)}\partial_{2}\tilde{g} \rangle \right| \lesssim E^{\frac{1}{4}}\\
 \left| \langle \tilde{\chi}_{1} \rangle  -  \langle \partial_{2}\tilde{g} \rangle \right| \lesssim E^{\frac{1}{4}}.\\ 
\end{array}
\right.
\end{align}
As multiplicative constants are irrelevant for the scaling behavior, we can ignore the factor $\frac{1}{4}$ in the expressions (\ref{eq:P5.1}), (\ref{eq:P5.2}).
As a consequence of (\ref{eq:P5.3}) and (\ref{eq:p2.2}) we can -- taking into account an error of $E^{\frac{1}{4}}$ -- work with the approximative quantities:
\begin{align*}
& \langle 1 + f_{(100)} + \partial_{2} \tilde{g} + f_{(100)}\partial_{2} \tilde{g}\rangle \left\langle \frac{1-f_{(100)}}{2} \right\rangle\\
& = \langle  1 -  f_{(100)}\partial_{2} \tilde{g} - f_{(100)} + \partial_{2}\tilde{g} \rangle \left\langle 1 - \frac{1 - f_{(100)}}{2} \right\rangle\\
 \Leftrightarrow & \;\; \langle 1 + \partial_{2} \tilde{g} \rangle \langle 1- f_{(100)} \rangle = \langle 1 - f_{(100)}\partial_{2} \tilde{g} - f_{(100)} + \partial_{2}\tilde{g} \rangle.
\end{align*}
Remembering $(\tilde{g}\circ\Phi)(s,t) = g(t)$ for $s\in[a,a+1]$, where $a = \frac{k}{2}$, $k\in \mathbb{Z}$, step 1 yields
\begin{align*}
 \langle 1 + \partial_{2}\tilde{g}\rangle \langle 1- f_{(100)} \rangle  
= \langle (1 + \partial_{2}\tilde{g})(1- f_{(100)}) \rangle.  
\end{align*}
\end{proof}

\bigskip

\subsection{Proof of the Third Statement}

In this section we show that we can not only prove weak rigidity in the sense of having approximately the right volume fractions, but that we can also obtain rigidity in a strong norm.
As we use the interpolation inequality again, we however give up a factor of $\eta^{-\frac{2}{9}}$ in the scaling behavior.\\
Thus, the question whether we need the BV control once more can be posed. As a counterexample (c.f. Lemma \ref{lem:counter}) at the end of the section proves, this is indeed a necessary condition. 
Yet, the counterexample does not imply that the use of the BV control enforces a loss in the scaling behavior.

\medskip

\begin{prop}
\label{prop: 5}
Let $\eta\leq 1$. Assume $\tilde{\chi}_{1}, \; \tilde{\chi}_{2}, \; \tilde{\chi}_{3}: \mathbb{T}^{2} \rightarrow \mathbb{R}$, $E_{elast}, \, E$ are as in Definition \ref{defi:FT}. Let 
\begin{align*}
 \Phi(s,t):= (s, t - F_{(100)}(s))
\end{align*}
with $F_{(100)}^{\prime}(s):= f_{(100)}(s)$ a.e., $F_{(100)}(0)=0$.\\
Then there exists $g: \Phi^{-1}([-\frac{1}{2},\frac{1}{2}]^{2})\rightarrow \mathbb{R}$, $(s,t)\mapsto g(t) $  such that
\begin{align*}
 \left\| \tilde{\chi}_{1}\circ\Phi - g \right\|^{2}_{L^{2}(\Phi^{-1}([-\frac{1}{2},\frac{1}{2}]^{2}))} \lesssim \eta^{-\frac{2}{9}}E^{\frac{5}{6}}
\end{align*}
and
\begin{align*}
 \left\| \tilde{\chi}_{2}\circ\Phi - (f_{(100)}\circ\Phi)g \right\|_{L^{2}(\Phi^{-1}([-\frac{1}{2},\frac{1}{2}]^{2}))} \lesssim  \max\left\{ \eta^{-\frac{1}{9}}E^{\frac{5}{12}}, E^{\frac{1}{4}} \right\}.
\end{align*}
\end{prop}

\bigskip

\begin{proof}
\textbf{Step 1:} \itshape We have
\begin{align}
\label{eq:P4.1}
 \left\| \nabla \cdot \begin{pmatrix}
                       \tilde{\chi}_{1}\\
		      -f_{(100)}\tilde{\chi}_{1}	
                      \end{pmatrix}
 \right\|_{H^{-1}(\mathbb{T}^{2})} \lesssim E^{\frac{1}{4}}.
\end{align}
\upshape 
This is a consequence of the estimate for $\tilde{\chi}_{3}$ as well as the characteristics. We have 
\begin{align}
\label{eq:P3.n}
 & \tilde{\chi}_{2} - \tilde{\chi}_{3}\tilde{\chi}_{1} = 0,\\
& \left\| \tilde{\chi}_{3} - f_{(100)} \right\|_{L^{2}([-\frac{1}{2}, \frac{1}{2}]^{2})} \lesssim E^{\frac{1}{4}}. \label{eq:P3.n1}
\end{align}

Combining these relations with Lemma \ref{lem:H-2per}, we obtain
\begin{align*}
  \left\| \nabla \cdot \begin{pmatrix}
                       \tilde{\chi}_{1}\\
		      -f_{(100)}\tilde{\chi}_{1}	
                      \end{pmatrix}
 \right\|_{H^{-1}(\mathbb{T}^{2})}
&= \left\|  \nabla \cdot \begin{pmatrix}
                       \tilde{\chi}_{1}\\
		      -f_{(100)}\tilde{\chi}_{1}	
                      \end{pmatrix} - \nabla \cdot \begin{pmatrix} \tilde{\chi}_{1}\\ -\tilde{\chi}_{3}\tilde{\chi}_{1}  \end{pmatrix} \right\|_{H^{-1}(\mathbb{T}^{2})}\\
& \;\;\;\;\;\; + \left\| \nabla \cdot \begin{pmatrix} \tilde{\chi}_{1}\\ -\tilde{\chi}_{3}\tilde{\chi}_{1}  \end{pmatrix} \right\|_{H^{-1}(\mathbb{T}^{2})}\\
& \stackrel{(\ref{eq:P3.n})}{\leq} \left\| \tilde{\chi}_{1}(-f_{(100)} + \tilde{\chi}_{3}) \right\|_{L^{2}(\mathbb{T}^{2})} + \left\| \partial_{1}\tilde{\chi}_{1} - \partial_{2}\tilde{\chi}_{2}  \right\|_{H^{-1}(\mathbb{T}^{2})}\\
& \stackrel{(\ref{eq:P3.n1})}{\lesssim} E^{\frac{1}{4}} + E_{elast}^{\frac{1}{2}},
\end{align*}
as Lemma \ref{lem:H-2per} states
\begin{align*}
& \left\| \partial_{1}\tilde{\chi}_{1} - \partial_{2}\tilde{\chi}_{2} \right\|_{H^{-1}(\mathbb{T}^{2})}
=  \left\| \nabla(\partial_{1}\tilde{\chi}_{1} - \partial_{2}\tilde{\chi}_{2}) \right\|_{H^{-2}(\mathbb{T}^{2})} \lesssim E_{elast}^{\frac{1}{2}}.
\end{align*}

\textbf{Step 2:} \itshape There exists $j_{\tau}, \; j: \Phi^{-1}([-\frac{1}{2}, \frac{1}{2}]^{2}) \rightarrow \mathbb{R}$ such that
\begin{align*}
& \partial_{s}^{h}\tilde{\chi}_{1}(\Phi(s,t)) = \partial_{t}j_{\tau}(s,t) + j(s,t),\\
& \int\limits_{\Phi^{-1}([-\frac{1}{2}, \frac{1}{2}]^{2})}{|j_{\tau}|^{2} + |j|^{2}dsdt} \lesssim E^{\frac{1}{2}}. 
\end{align*}
\upshape The characterization of the $H^{-1}$-norm implies the existence of $\rho_{1}, \rho_{2}: \mathbb{T}^{2} \rightarrow \mathbb{R}$ such that
\begin{align}
\label{eq:P312}
 &\partial_{1}\tilde{\chi}_{1} - \partial_{2}(f_{(100)}\tilde{\chi}_{1}) = \partial_{1}\rho_{1} + \partial_{2}\rho_{2} \mbox{ in } \mathcal{D}^{\prime},\\
& \int\limits_{\mathbb{T}^{2}}\rho_{1}^{2} + \rho_{2}^{2}dx \lesssim E^{\frac{1}{2}}.\nonumber 
\end{align}
Due to the structure of the change of coordinates $\Phi$, a distributional chain rule holds:
\begin{align*}
 \partial_{s}(\tilde{\chi}_{1}\circ\Phi) = \partial_{s}(\rho_{\sigma}\circ\Phi) + \partial_{t}(\rho_{\tau}\circ\Phi),
\end{align*}
for $\rho_{\sigma} = \rho_{1}$ and $\rho_{\tau}= f_{(100)}\rho_{1} + \rho_{2}$. Consequently we have:
\begin{align}
\label{eq:P31}
 \int\limits_{\Phi^{-1}(\mathbb{T}^{2})}|\rho_{\sigma}(\Phi(s,t))|^{2} + |\rho_{\tau}(\Phi(s,t))|^{2}dsdt \lesssim E^{\frac{1}{2}}.
\end{align}
As sums and products of $\left[ -\frac{1}{2}, \frac{1}{2} \right]^{2}$-periodic functions, $\rho_{\sigma}$ and $\rho_{\tau}$ are periodic as well.
A calculation as in Lemma \ref{lem:OS-1} converts this into
\begin{align*}
 \partial_{s}^{h}((\tilde{\chi}_{1}\circ\Phi) - (\rho_{\sigma}\circ\Phi))(s,t) 
& = \int\limits_{0}^{h}\frac{d}{dh^{\prime}}((\tilde{\chi}_{1}\circ\Phi) - (\rho_{\sigma}\circ\Phi))(s + h^{\prime},t)dh^{\prime}\\
& =  \frac{d}{dt}\int\limits_{0}^{h}{\rho_{\tau}(\Phi(s + h^{\prime}, t))dh^{\prime}}.
\end{align*}
Setting
\begin{align*}
 & j(s,t) := \partial_{s}^{h}\rho_{\sigma}(\Phi(s,t)),\\
& j_{\tau}(s,t):= \int\limits_{0}^{h}\rho_{\tau}(\Phi(s+h^{\prime},t))dh^{\prime},
\end{align*}
and using (\ref{eq:P31}), we obtain
\begin{align*}
& \partial_{s}^{h}\tilde{\chi}_{1}(\Phi(s,t)) = \partial_{t}j_{\tau}(s,t) + j(s,t),\\
& \int\limits_{\Phi^{-1}([-\frac{1}{2}, \frac{1}{2}]^{2})}{|j_{\tau}|^{2} + |j|^{2}dsdt} \lesssim E^{\frac{1}{2}}. 
\end{align*}

\textbf{Step 3:} \itshape Interpolation inequality. \\
\upshape We use the interpolation inequality of Lemma \ref{lem:interpol} applied to (the periodic function) $\phi = (\partial_{s}^{h}(\tilde{\chi}_{1}\circ\Phi))\circ\Phi^{-1}$
in its multiplicative version and carry out a change of coordinates to derive:
\begin{multline*}
\int\limits_{\Phi^{-1}(\mathbb{T}^{2})}{|\partial_{s}^{h}\tilde{\chi}_{1}(\Phi(s,t))|^{2}dsdt} \\
\begin{aligned}
& \leq \left(\int\limits_{\Phi^{-1}(\mathbb{T}^{2})}{|\nabla_{(s,t)}(\partial_{s}^{h}\tilde{\chi}_{1}(\Phi(s,t)))|dsdt }\sup|(\partial_{s}^{h}(\tilde{\chi}_{1}\circ\Phi))\circ\Phi^{-1}|\right)^{\frac{2}{3}}\\
& \;\;\;\;\; \times \left( \int\limits_{\Phi^{-1}(\mathbb{T}^{2})}{|j_{\tau}|^{2}+|j|^{2}}dsdt \right)^{\frac{1}{3}}.
\end{aligned}
\end{multline*}
More precisely, the multiplicative version is first used on the torus with the periodic functions 
$$ (\partial_{s}^{h}(\tilde{\chi}_{1}\circ\Phi))\circ\Phi^{-1}, \; j\circ\Phi^{-1}, \; j_{\tau}\circ\Phi^{-1}. $$
In a second step the transformation rule is applied together with the volume conservation of $\Phi$ which then yields the desired result.\\
In order to convince oneself of the periodicity of $(\partial_{s}^{h}(\tilde{\chi}_{1}\circ\Phi))\circ\Phi^{-1}, \; j\circ\Phi^{-1}, \; j_{\tau}\circ\Phi^{-1}$ we argue by 
inserting the change of coordinates:
\begin{align*}
 ((\partial_{s}^{h}(\tilde{\chi}_{1}\circ\Phi))\circ\Phi^{-1})(y_{1},y_{2}) =& \; \tilde{\chi}_{1}(y_{1} + h,y_{2}-F_{(100)}(y_{1}+h) +F_{(100)}(y_{1}))\\
& - \tilde{\chi}_{1}(y_{1},y_{2}).
\end{align*}
Here $F_{(100)}$ denotes the normalized antiderivative of $f_{(100)}$ (i.e. $F_{(100)}(0)=0$). Since $f_{(100)}$ is the mean value of a periodic function it is periodic itself.
Hence, the fundamental theorem of calculus implies that $F_{(100)}(y_{1}+h) - F_{(100)}(y_{1})$ inherits this property. Combined with the periodicity of $\tilde{\chi}_{1}$ 
this implies the periodicity of $(\partial_{s}^{h}(\tilde{\chi}_{1}\circ\Phi))\circ\Phi^{-1}$.\\
The periodicity of $j\circ\Phi^{-1}$ follows from the periodicity of $\rho_{\sigma}$.\\
A similar argument works for $j_{\tau}\circ\Phi^{-1}$: Due to
\begin{align*}
 (j_{\tau}\circ\Phi^{-1})(y_{1},y_{2}) = \int\limits_{0}^{h}{\rho_{\tau}(y_{1} + h^{\prime}, y_{2} - F_{(100)}(y_{1}+h^{\prime}) + F_{(100)}(y_{1}))dh^{\prime}},
\end{align*}
the periodicity of $F_{(100)}(y_{1}+h^{\prime}) - F_{(100)}(y_{1})$ and $\rho_{\tau}$ implies the claim.\\
Using an approximation argument the first term in the interpolation inequality can be estimated by the homogeneous BV norm of the martensitic phases.
Together with the $L^{\infty}$-bound for $\tilde{\chi}_{1}$, this yields the desired result:
\begin{multline*}
 \int\limits_{\Phi^{-1}(\mathbb{T}^{2})}{|\partial_{s}^{h}\tilde{\chi}_{1}(\Phi(s,t))|^{2}dsdt} \\
\begin{aligned}
& \leq \left(\int\limits_{\Phi^{-1}(\mathbb{T}^{2})}{|\nabla_{(s,t)}(\partial_{s}^{h}\tilde{\chi}_{1}(\Phi(s,t)))|dsdt }\sup|(\partial_{s}^{h}(\tilde{\chi}_{1}\circ\Phi))\circ\Phi^{-1}|\right)^{\frac{2}{3}}
\\
&\;\;\;\;\;\; \times \left( \int\limits_{\Phi^{-1}(\mathbb{T}^{2})}{|j_{\tau}|^{2}+|j|^{2}}dsdt \right)^{\frac{1}{3}}\\
& \lesssim \left(\int\limits_{\mathbb{T}^{2}}{|\nabla\chi_{1}| +|\nabla \chi_{4}|dx }\right)^{\frac{2}{3}}\left( \int\limits_{\Phi^{-1}(\mathbb{T}^{2})}{|j_{\tau}|^{2}+|j|^{2}}dsdt \right)^{\frac{1}{3}}\\
& \lesssim (E_{surf})^{\frac{2}{3}}E^{\frac{1}{6}} \lesssim (\eta^{-\frac{1}{3}}E)^{\frac{2}{3}}E^{\frac{1}{6}}\lesssim \eta^{-\frac{2}{9}}E^{\frac{5}{6}}.
\end{aligned}
\end{multline*}
Finally, the $L^{1}$-estimate follows from the discreteness of $\tilde{\chi}_{1}$.

\textbf{Step 4:} \itshape $L^{1}$-control.\\
There exists $g: \Phi^{-1}([-\frac{1}{2}, \frac{1}{2}]^{2}) \rightarrow \mathbb{R}$, $(s,t) \mapsto g(t)$ such that
\begin{align*}
& \int\limits_{\Phi^{-1}([-\frac{1}{2}, \frac{1}{2}]^{2})}|\tilde{\chi}_{1}(\Phi(s,t)) - g(t)|dsdt\\
& \lesssim \sup\limits_{|h|\leq 1}\int\limits_{\Phi^{-1}([-\frac{1}{2},\frac{1}{2}]^{2})}|\partial_{s}^{h}\tilde{\chi}_{1}(\Phi(s,t))|dsdt \lesssim \eta^{-\frac{2}{9}}E^{\frac{5}{6}}.
\end{align*}
\upshape As $\Phi$ is a bilipschitz mapping, we find $r\geq 1$ and $ k \in \mathbb{N}$, $k\geq 2$ such that the following inclusions are satisfied  $$\Phi^{-1}\left(\left[-\frac{1}{2},\frac{1}{2}\right]^{2}\right) \subset \left[ -\frac{r}{2}, \frac{r}{2} \right]^{2} \subset \left[ -r, r \right]^{2} \subset \Phi^{-1}\left(\left[-\frac{k}{2}, \frac{k}{2}\right]^{2}\right).$$
Thus, the claim is a consequence of the following estimates:
\begin{multline*}
 \sup\limits_{|h|\leq 1}\int\limits_{\Phi^{-1}([-\frac{k}{2}, \frac{k}{2}]^{2})}|\tilde{\chi}_{1}(\Phi(s+h,t)) - \tilde{\chi}_{1}(\Phi(s,t))|dsdt\\
\begin{aligned}
& \geq \frac{1}{2}\int\limits_{-1}^{1}\int\limits_{\left[ -r, r \right]^{2} }|\tilde{\chi}_{1}(\Phi(s+h,t)) - \tilde{\chi}_{1}(\Phi(s,t))|dsdtdh\\
& \gtrsim  \int\limits_{\left[ -\frac{r}{2}, \frac{r}{2} \right]^{2} }\int\limits_{-\frac{1}{2}}^{\frac{1}{2}}|\tilde{\chi}_{1}(\Phi(y,t)) -\tilde{\chi}_{1}(\Phi(s,t))|dydsdt\\
& \gtrsim \int\limits_{\left[ -\frac{r}{2}, \frac{r}{2} \right]^{2} }\left|\tilde{\chi}_{1}(\Phi(s,t)) - \int\limits_{-\frac{1}{2}}^{\frac{1}{2}}\tilde{\chi}_{1}(\Phi(y,t))dy \right|dsdt\\
& \gtrsim   \int\limits_{\Phi^{-1}([-\frac{1}{2}, \frac{1}{2}]^{2})}|\tilde{\chi}_{1}(\Phi(s,t)) - g(t)|dsdt,
\end{aligned}
\end{multline*}
where $g(t)= \int\limits_{-\frac{1}{2}}^{\frac{1}{2}}\tilde{\chi}_{1}(\Phi(y,t))dy $. \\
Since
\begin{align*}
  &\sup\limits_{|h|\leq 1}\int\limits_{\Phi^{-1}([-\frac{k}{2}, \frac{k}{2}]^{2})}|\partial_{s}^{h}\tilde{\chi}_{1}(\Phi(s,t))|dsdt\\
&\leq Ck^{2}\sup\limits_{|h|\leq 1}\int\limits_{\Phi^{-1}([-\frac{1}{2},\frac{1}{2}]^{2})}|\partial_{s}^{h}\tilde{\chi}_{1}(\Phi(s,t))|dsdt \lesssim \eta^{-\frac{2}{9}}E^{\frac{5}{6}},
\end{align*}
the claim is proven in the $L^{1}$-topology.\\
Due to the discreteness/boundedness of all quantities, we also obtain the $L^{2}$-estimate for $\tilde{\chi}_{1}$.\\

\textbf{Step 5:} \itshape Estimate for $\tilde{\chi}_{2}$.\\
\upshape Let $g:\Phi^{-1}([-\frac{1}{2}, \frac{1}{2}]^{2}) \rightarrow \mathbb{R}$ be the function from step 4. The estimate for $\tilde{\chi}_{1}$ in combination with the identity
 $\tilde{\chi}_{2} = \tilde{\chi}_{3}\tilde{\chi}_{1}$ yields:
\begin{multline*}
 \left\| \tilde{\chi}_{2}\circ\Phi - (f_{(100)}\circ\Phi)g \right\|_{L^{2}(\Phi^{-1}([-\frac{1}{2}, \frac{1}{2}]^{2}))} \\
\begin{aligned}
& = \left\| (\tilde{\chi}_{3}\circ\Phi)(\tilde{\chi}_{1}\circ\Phi) - (f_{(100)}\circ\Phi)g \right\|_{L^{2}(\Phi^{-1}([-\frac{1}{2}, \frac{1}{2}]^{2}))} \\
& \leq \left\| (\tilde{\chi}_{3}\circ\Phi)((\tilde{\chi}_{1}\circ\Phi) - g) \right\|_{L^{2}(\Phi^{-1}([-\frac{1}{2}, \frac{1}{2}]^{2}))}\\
& \;\;\;\;\;+ \left\| g((f_{(100)}\circ\Phi)- \tilde{\chi}_{3}\circ\Phi) \right\|_{L^{2}(\Phi^{-1}([-\frac{1}{2}, \frac{1}{2}]^{2}))}\\
& \leq \left\| \tilde{\chi}_{3}\circ\Phi \right\|_{L^{\infty}(\Phi^{-1}([-\frac{1}{2}, \frac{1}{2}]^{2})))} \left\|\tilde{\chi}_{1}\circ\Phi - g \right\|_{L^{2}(\Phi^{-1}([-\frac{1}{2}, \frac{1}{2}]^{2}))}\\
& \;\;\;\;\; + \left\| g\right\|_{L^{\infty}(\Phi^{-1}([-\frac{1}{2}, \frac{1}{2}]^{2})))} \left\|(f_{(100)}\circ\Phi)- \tilde{\chi}_{3}\circ\Phi \right\|_{L^{2}(\Phi^{-1}([-\frac{1}{2}, \frac{1}{2}]^{2}))}\\
& \lesssim \eta^{-\frac{1}{9}}E^{\frac{5}{12}} + E^{\frac{1}{4}}.
\end{aligned}
\end{multline*}

This proves the claim for $\tilde{\chi}_{2}$.
\end{proof}

\medskip

Finally, we present the counterexample stating that we have to use the BV control in order to obtain strong rigidity.

\begin{lem}
\label{lem:counter}
Let $f_{(100)}: \mathbb{T}^{2} \rightarrow \{-1, 1\}$, $(s,t)\mapsto f_{(100)}(s)$. Then there exists a sequence $\left\{ ( \tilde{\chi}_{1}^{(k)}, \tilde{\chi}_{2}^{(k)}, \tilde{\chi}_{3}^{(k)}) \right\}_{k\in\mathbb{N}}, \tilde{\chi}_{i}^{(k)}: \mathbb{T}^{2} \rightarrow \{-1,1\}$, 
and corresponding sequences $\left\{ u^{(k)} \right\}_{k \in \mathbb{N}}$, $u^{(k)}:\mathbb{T}^{2}\rightarrow \mathbb{R}$ and $\{g^{(k)}\}_{k \in \mathbb{N}}, \; g^{(k)}:\mathbb{T}^{2} \rightarrow \mathbb{R}$, $(s,t) \mapsto g^{(k)}(t)$ such that  
\begin{align}
\label{eq:G1}
& \left\| \begin{pmatrix} \tilde{\chi}_{2}^{(k)} - f_{(100)}\tilde{\chi}_{1}^{(k)} \\ \tilde{\chi}_{1}^{(k)} \end{pmatrix} - \begin{pmatrix} \partial_{s}u^{(k)} \\ \partial_{t}u^{(k)} \end{pmatrix} \right\|_{L^{2}(\mathbb{T}^{2})} \leq \frac{1}{k},\\
\label{eq:G2}
& \left\| u^{(k)} - g^{(k)} \right\|_{L^{2}(\mathbb{T}^{2})} \leq \frac{1}{k}, \\
\label{eq:G3}
&\left\| \partial_{s}u^{(k)} \right\|_{L^{2}(\mathbb{T}^{2})} \leq \frac{1}{k}, \\
\label{eq:G4}
&\tilde{\chi}_{2}^{(k)} - \tilde{\chi}_{3}^{(k)}\tilde{\chi}_{1}^{(k)} = 0, \\
& \tilde{\chi}_{3}^{(k)}= f_{(100)},
\end{align}
but
\begin{align*}
 \left\| \tilde{\chi}_{1}^{(k)} - f  \right\|_{L^{2}(\mathbb{T}^{2})}\geq C
\end{align*}
for all $f:\mathbb{T}^{2} \rightarrow \mathbb{R}$, $(s,t) \mapsto f(t)$. Here $C>0$ is a universal constant.
\end{lem}

\medskip

\begin{proof}
In order to define the sequence $\tilde{\chi}_{i}^{(k)}$, we first construct a sample which is then rescaled appropriately.
For that purpose, consider the function $u:\mathbb{T}^{2}\rightarrow \mathbb{R}$, whose gradient is depicted in Figure \ref{fig:grad} on the interval $\left[ -\frac{1}{2}, \frac{1}{2} \right] ^{2}$
and for which we have $u=0$ on $\left\{ s=t \right\} \cap \left[ -\frac{1}{2}, 0 \right]^{2}$ and on $\left\{ t=-s \right\} \cap \left\{ \left[ 0, \frac{1}{2} \right]\times\left[ -\frac{1}{2} , 0 \right] \right\}$.
This is to be continued periodically.

\begin{figure}[htbp]
\centering
\psset{unit=4.7cm}
 \begin{pspicture}(-0.75,-0.8)(0.75,0.8)
\psline{->}(-0.75,0)(.75,0)
\psline{->}(0,-.75)(0,.75)
\psline(-0.5,-0.5)(0,0)
\psline(0,0)(0.5,-0.5)
\psline(-0.5,0)(0,0.5)
\psline(0,0.5)(0.5,0)
\psline[linestyle=dotted](-0.5,-.75)(-0.5,.75)
\psline[linestyle=dotted](0.5,-.75)(0.5,.75)
\psline[linestyle=dotted](-0.75,-0.5)(0.75,-0.5)
\psline[linestyle=dotted](-0.75,0.5)(0.75,0.5)  
\put(0.05,-0.4){$\begin{pmatrix} 1 \\ 1 \end{pmatrix}$}
\put(-0.25,-0.4){$\begin{pmatrix} -1 \\ 1 \end{pmatrix}$}
\put(0.15,-0.03){$\begin{pmatrix} -1 \\ -1 \end{pmatrix}$}
\put(-0.35,-0.03){$\begin{pmatrix} 1 \\ -1 \end{pmatrix}$}
\put(-0.48,0.3){$\begin{pmatrix} -1 \\ 1 \end{pmatrix}$}
\put(0.3,0.3){$\begin{pmatrix} 1 \\ 1 \end{pmatrix}$}
\put(0.5,-0.1){$\frac{1}{2}$}
\psline(0.5,-0.04)(0.5,0.04)
\put(0.075,0.5){$\frac{1}{2}$}
\psline(-0.04,0.5)(0.04,0.5)
\put(0.78,-0,1){$s$}
\put(0.1, 0.76){$t$}
 \end{pspicture}
\caption{On the torus $ \partial_{s}u $ is depicted in the first and $ \partial_{t}u$ in the second component.}
\label{fig:grad}
\end{figure}
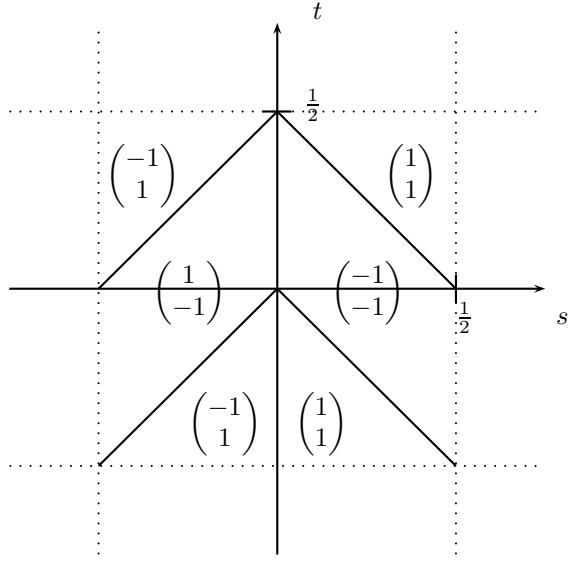

On $\left[ -\frac{1}{2}, \frac{1}{2} \right]\times \left[ -\frac{1}{4}, -\frac{3}{8} \right]$ the $s$-average of the $t$-derivative is given by
\begin{align*}
 \int\limits_{-\frac{1}{2}}^{\frac{1}{2}}{\partial_{t}u(s, \cdot)ds} = -4t -1.
\end{align*}
Thus, the deviation from the average can be estimated by
\begin{align}
\label{eq:G5}
\int\limits_{\left[ -\frac{1}{2}, \frac{1}{2} \right]^{2}}\left| \partial_{t}u - \int\limits_{-\frac{1}{2}}^{\frac{1}{2}} \partial_{t}u(s^{\prime},\cdot)ds^{\prime} \right|^{2}dsdt &\geq \int\limits_{\left[ -\frac{1}{4}, \frac{1}{4} \right]\times \left[ -\frac{1}{4}, -\frac{3}{8} \right]}{\left|1 + 4\left(t+\frac{1}{4}\right)\right|^{2}dsdt} \nonumber \\
 &\geq \frac{1}{2}\cdot\frac{1}{8}\cdot\frac{1}{4} = \frac{1}{64} > 0.
\end{align}
We rescale and set
\begin{align*}
 u^{(k)}(s,t) := \frac{1}{k^{2}}u(ks,k^{2}t), \; k\in\mathbb{N}.
\end{align*}
Further we define
\begin{align*}
& \tilde{\chi}_{1}^{(k)} := \partial_{t}u^{(k)} \in \{ -1,1 \},\\
& \tilde{\chi}_{2}^{(k)} := f_{(100)}\partial_{t}u^{(k)} \in \{ -1,1 \},\\
& \tilde{\chi}_{3}^{(k)} := f_{(100)}.
\end{align*}
Hence, (\ref{eq:G1}), (\ref{eq:G2}), (\ref{eq:G3}), (\ref{eq:G4}) are satisfied since
\begin{align*}
 \left\| \begin{pmatrix} \tilde{\chi}_{2}^{(k)} - f_{(100)}\tilde{\chi}_{1}^{(k)}\\ \tilde{\chi}_{1}^{(k)} \end{pmatrix} - \begin{pmatrix} \partial_{s}u^{(k)} \\ \partial_{t}u^{(k)} \end{pmatrix} \right\|_{L^{2}(\mathbb{T}^{2})}
 = \left\| \partial_{s}u^{(k)} \right\|_{L^{2}(\mathbb{T}^{2})} \leq \frac{1}{k}, 
\end{align*}
and since due to $|u| \leq \frac{1}{2}$, we have $|u^{(k)}| \leq \frac{1}{k^{2}}$. Therefore, $\left\| u^{(k)} \right\|_{L^{2}(\mathbb{T}^{2})} \leq \frac{1}{k^{2}}.$ Due to (\ref{eq:G5}), we obtain
\begin{align*}
 \int\limits_{\left[-\frac{1}{2}, \frac{1}{2} \right]^{2}}{|\tilde{\chi}^{(k)}_{1} -  \int\limits_{-\frac{1}{2}}^{\frac{1}{2}} \tilde{\chi}_{1}^{(k)}(s^{\prime},\cdot)ds^{\prime}|^{2}dsdt} \geq \frac{1}{64}.
\end{align*}
Finally, this leads to
\begin{align*}
 \inf\limits_{f(t)\in L^{2}([-\frac{1}{2},\frac{1}{2}])}\left\| \tilde{\chi}_{1}^{(k)} -f \right\|^{2}_{L^{2}(\mathbb{T}^{2})} \geq \frac{1}{64} >0,
\end{align*}
as taking the mean value corresponds to the $L^{2}$-projection on the space of constants.
\end{proof}

\bigskip

\subsection{Optimality and Branching}

In this final section we give an upper bound construction to derive the optimal scaling behavior of incompatible microstructures. Combined with Theorem \ref{thm:2}, this proves
Proposition \ref{cor:1}.
The construction relies on an idea introduced in \cite{CO08}.

\medskip

\begin{lem}
\label{lem:branch}
Let $\eta \ll 1$, consider
\begin{align*}
\MoveEqLeft{ E_{\eta}(e,\chi) :=  \eta^{\frac{1}{3}}\int\limits_{(0,1)^{2}}{|\nabla \chi_{1}| + |\nabla \chi_{2}| + |\nabla \chi_{3}| + |\nabla \chi_{4}|dy}}\\
	&	 +  \eta^{-\frac{2}{3}}\int\limits_{(0,1)^{2}}{\left| e - \begin{pmatrix} 
		                                                           d_{1} & \tilde{\chi}_{3} & \tilde{\chi}_{2}\\
									   \tilde{\chi}_{3} & d_{2} & \tilde{\chi}_{1}\\
									   \tilde{\chi}_{2} & \tilde{\chi}_{1} & d_{3}
		                                                          \end{pmatrix} \right|^{2}dy}.
\end{align*}
For $\delta = \frac{3}{16}$ there exists a family of $(0,1)^{2}$-periodic strain tensors
$\{e_{\eta}\}_{\eta}\in\Sym(3,\mathbb{R})$, $e_{\eta} = e_{\eta}(y_{1},y_{2})$, and a family $\{ \chi^{\eta} =(\tilde{\chi}_{2}^{\eta}, \tilde{\chi}_{3}^{\eta}, \tilde{\chi}_{1}^{\eta}) \}_{\eta}$,
$\tilde{\chi}_{i}\in\{-1,1\}$ such that
\begin{equation*}
 E_{\eta}(e_{\eta},\chi^{\eta}) \lesssim 1,
\end{equation*}
and
\begin{align}
\label{eq:Mikr1}
& |\theta^{\eta}_{1}(\theta^{\eta}_{2} + \theta^{\eta}_{4}) - \theta^{\eta}_{4}(\theta^{\eta}_{1}+\theta^{\eta}_{3})| \geq \delta, \\
\label{eq:Mikr2}
& |\theta^{\eta}_{1}(\theta^{\eta}_{2} + \theta^{\eta}_{4}) - \theta^{\eta}_{2}(\theta^{\eta}_{1}+\theta^{\eta}_{3})| \geq \delta.
\end{align}
\end{lem}

\medskip

\begin{proof}																		
\textbf{Step 1:} \itshape Choice of $e_{\eta}.$\\
\upshape  
Consider
\begin{align*}
 E_{\eta}(\chi) := \; & \eta^{\frac{1}{3}}\int\limits_{(0,1)^{2}}{|\nabla {\chi}_{1}| + |\nabla {\chi}_{2}| + |\nabla {\chi}_{3}| + |\nabla {\chi}_{4}|dx}\\
& + \eta^{-\frac{2}{3}}\inf\limits_{\nabla u \mbox{ periodic}} \left\{ \int\limits_{(0,1)^{2}}{
\left| \frac{\nabla u + (\nabla u)^{t}}{2} - 
\begin{pmatrix}
 d_{1} & \tilde{\chi}_{3} & \tilde{\chi}_{2} \\
\tilde{\chi}_{3} & d_{2} & \tilde{\chi}_{1}\\
\tilde{\chi}_{2} & \tilde{\chi}_{1} & d_{3}
\end{pmatrix}
 \right|^{2}
dx} \right\}\\ 
& =: \eta^{\frac{1}{3}} E_{surf}(\chi) + \eta^{-\frac{2}{3}} E_{elast}(\chi) . 
\end{align*}
For each $\chi$ and each $\eta\in(0,1)$ there exists $e_{\eta}$ such that
\begin{align*}
 |E_{\eta}(\chi) - E_{\eta}(e_{\eta},\chi)| \leq E_{\eta}(\chi),
\end{align*}
which especially implies
\begin{align*}
 E_{\eta}(e_{\eta},\chi) \leq 2 E_{\eta}(\chi).
\end{align*}
Thus, it suffices to estimate $E_{\eta}(\chi)$ in the sequel.\\

\textbf{Step 2:} \itshape Calculation of the Fourier multiplier for the elastic energy.\\
\upshape
In order to obtain the Fourier multiplier of the elastic energy the Euler-Lagrange-Equations of the elastic energy have to be computed.  
These are well known (c.f. \cite{CO08}) and lead to 
\begin{align*}
 E_{elast}(\chi) = \sum\limits_{\substack{ k\in \mathbb{Z}^{2} \\ k\neq 0}}{|k|^{-4}(|k|^{4}|\mathcal{F}{\tilde{U}}_{0}|^{2} - 2|k|^{2}|\mathcal{F}{\tilde{U}}_{0}k|^{2} + |k\cdot \mathcal{F}{\tilde{U}}_{0}k|^{2})},
\end{align*}
where
\begin{align*}
 \tilde{U}_{0}=\begin{pmatrix}
 d_{1} & \tilde{\chi}_{3} & \tilde{\chi}_{2} \\
\tilde{\chi}_{3} & d_{2} & \tilde{\chi}_{1}\\
\tilde{\chi}_{2} & \tilde{\chi}_{1} & d_{3}
\end{pmatrix}.
\end{align*}
As the Fourier multiplier ignores the mode $k=0$, i.e. constants, we can use 
\begin{align*}
 U_{0}:= \begin{pmatrix}
 0 & \tilde{\chi}_{3} & \tilde{\chi}_{2} \\
\tilde{\chi}_{3} & 0 & \tilde{\chi}_{1}\\
\tilde{\chi}_{2} & \tilde{\chi}_{1} & 0
\end{pmatrix}
\end{align*}
instead. Consequently the summands in the multiplier can be computed as
\begin{align*}
&|k|^{4}|\mathcal{F}{U}_{0}|^{2} = 2(k_{1}^{4} + k_{2}^{4} + 2k_{1}^{2}k_{2}^{2} )(|\mathcal{F}{\tilde{\chi}}_{2}|^{2} + |\mathcal{F}{\tilde{\chi}}_{3}|^{2} + |\mathcal{F}{\tilde{\chi}}_{1}|^{2}),\\
& |k|^{2}|\mathcal{F}{U}_{0}k|^{2} =  (k_{1}^{2} + k_{2}^{2})(k_{1}^{2} |\mathcal{F}{\tilde{\chi}}_{2}|^{2} + (k_{1}^{2} + k_{2}^{2})|\mathcal{F}{\tilde{\chi}}_{3}|^{2} \\ 
&  \;\;\;\;\;+k_{2}^{2}|\mathcal{F}{\tilde{\chi}}_{1}|^{2} + 2 Re(k_{1}\overline{\mathcal{F}{\tilde{\chi}}_{2}}k_{2}\mathcal{F}{\tilde{\chi}}_{1})),\\
& (k\cdot \mathcal{F}{U}_{0}k)^{2} = 4 k_{1}^{2}k_{2}^{2}|\mathcal{F}{\tilde{\chi}}_{3}|^{2}.
\end{align*}

Therefore the multiplier is determined by the following expressions:
\begin{align}
\MoveEqLeft{ |k|^{-4}2( k_{1}^{4}|\mathcal{F}{\tilde{\chi}}_{1}|^{2} + k_{2}^{4}|\mathcal{F}{\tilde{\chi}}_{2}|^{2} 
 + k_{1}^{2}k_{2}^{2}|\mathcal{F}{\tilde{\chi}}_{2}|^{2} + 2k_{1}^{2}k_{2}^{2}|\mathcal{F}{\tilde{\chi}}_{3}|^{2} 
+ k_{1}^{2}k_{2}^{2}|\mathcal{F}{\tilde{\chi}}_{1}|^{2}}\nonumber \\
& - 2k_{1}^{2}Re(k_{1}\overline{\mathcal{F}{\tilde{\chi}}_{2}}k_{2}\mathcal{F}{\tilde{\chi}}_{1}) - 2k_{2}^{2}Re(k_{1}\overline{\mathcal{F}{\tilde{\chi}}_{2}}k_{2}\mathcal{F}{\tilde{\chi}}_{1}))\nonumber \\
\label{eq:Fourier}
&=
2|k|^{-4}(|k|^{2}|k_{2}\mathcal{F}\tilde{\chi}_{2} - k_{1}\mathcal{F}\tilde{\chi}_{1}|^{2} + 2k_{1}^{2}k_{2}^{2}|\mathcal{F}\tilde{\chi}_{3}|^{2}).
\end{align}

\textbf{Step 3: }\itshape Introduction of the quantities involved in the construction.\\
\upshape 
For $\lambda \in [0,1]$, $(y_{1},y_{2})\in [0,1]^{2}$ we define $\sigma^{\eta}_{(\lambda,1)}(y_{1},y_{2})$ with 
\begin{align*}
 \sigma^{\eta}_{(\lambda,1)}(y_{1},y_{2}) := \left\{ \begin{array}{ll}
                        \in\{-1,1\}, & y_{2}\in(\lambda,1),\\
			   = 1-2\mu, &\mbox{else.}	 
                         \end{array} \right.
\end{align*}
For this purpose, we consider the following quantities:
\begin{align*}
& \mu, \; \lambda \in [0,1],\\
& N, w_{1}^{-1}\in \mathbb{N},\\
& w_{n} := 2^{-(n-1)}w_{1},\\
& l_{n} := 2^{-\beta(n-1)}l_{1}, \, \beta>0, \; n\in \{1,\cdots, N\},
\end{align*}
where we choose $l_{1}$ such that $\sum\limits^{N}_{n=1}{l_{n}} = \frac{1-\lambda}{2}.$
On $[0,1]\times [\lambda,1]$ we introduce a construction of rows of self-similar cells of width $w_{n}$ and height $l_{n}$ as depicted in Figure \ref{fig:branch}. Within each of the cells
we define $\sigma^{\eta}_{(\lambda,1)}$ as indicated in Figure \ref{fig:sigma}.

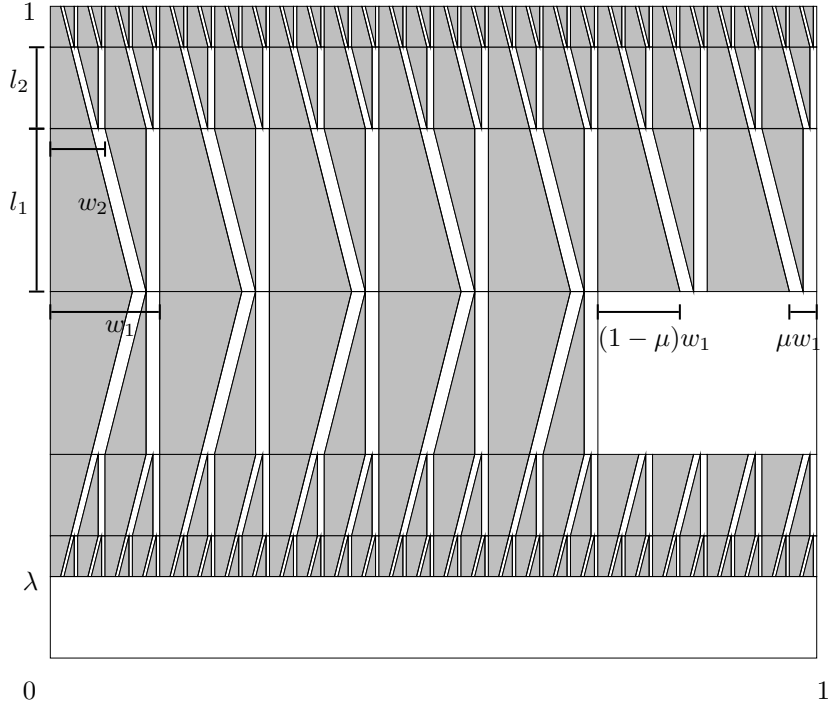
\begin{figure}[h]
 \begin{pspicture}(-1,-5.6)(4,3)
  \psset{unit=0.18cm, 
linewidth=0.3pt}
\psset{yunit=1.5}
\put(-2,-30){0}
\put(56,-30){1}
\put(-2,20){1}
\put(-2,-22){$\lambda$}
\psline[linewidth=0.75pt]{|*-|*}(-1,0)(-1,8)
\put(-3,6){$l_{1}$}
\psline[linewidth=0.75pt]{|*-|*}(-1,8)(-1,12)
\put(-3,15){$l_{2}$}
\pspolygon(0,-14)(0,-18)(56,-18)(56,-14)
  \multirput(0,0)(8,0){7}{\pspolygon[fillstyle=solid,
    fillcolor=lightgray]
  (0,0)(6,0)(3,8)(0,8)\pspolygon(6,0)(7,0)(4,8)(3,8)\pspolygon[fillstyle=solid,
    fillcolor=lightgray](7,0)(7,8)(4,8)\pspolygon(7,0)(8,0)(8,8)(7,8)}
 \multirput(0,0)(8,0){5}{\pspolygon[fillstyle=solid,
    fillcolor=lightgray]
  (0,0)(6,0)(3,-8)(0,-8)\pspolygon(6,0)(7,0)(4,-8)(3,-8)\pspolygon[fillstyle=solid,
    fillcolor=lightgray](7,0)(7,-8)(4,-8)\pspolygon(7,0)(8,0)(8,-8)(7,-8)}
\psline(56,0)(56,-8)
\psline[linewidth= 0.75pt]{|*-|*}(0,-1)(8,-1)
\put(4,-3){$w_{1}$}
\psline[linewidth= 0.75pt]{|*-|*}(0,7)(4,7)
\put(2,6){$w_{2}$}
\psline[linewidth= 0.75pt]{|*-|*}(54,-1)(56,-1)
\put(53,-4){$\mu w_{1}$}
\psline[linewidth= 0.75pt]{|*-|*}(40,-1)(46,-1)
\put(40,-4){$(1-\mu) w_{1}$}
\psset{unit=0.5}
\psset{yunit=1.5}
  \multirput(0,16)(8,0){14}{\pspolygon[fillstyle=solid,
    fillcolor=lightgray]
  (0,0)(6,0)(3,8)(0,8)\pspolygon(6,0)(7,0)(4,8)(3,8)\pspolygon[fillstyle=solid,
    fillcolor=lightgray](7,0)(7,8)(4,8)\pspolygon(7,0)(8,0)(8,8)(7,8)}
 \multirput(0,-16)(8,0){14}{\pspolygon[fillstyle=solid,
    fillcolor=lightgray]
  (0,0)(6,0)(3,-8)(0,-8)\pspolygon(6,0)(7,0)(4,-8)(3,-8)\pspolygon[fillstyle=solid,
    fillcolor=lightgray](7,0)(7,-8)(4,-8)\pspolygon(7,0)(8,0)(8,-8)(7,-8)}
\psset{unit=0.5, linewidth = 0.1pt}
\psset{yunit=1.5}
  \multirput(0,48)(8,0){28}{\pspolygon[fillstyle=solid,
    fillcolor=lightgray]
  (0,0)(6,0)(3,8)(0,8)\pspolygon(6,0)(7,0)(4,8)(3,8)\pspolygon[fillstyle=solid,
    fillcolor=lightgray](7,0)(7,8)(4,8)\pspolygon(7,0)(8,0)(8,8)(7,8)}
 \multirput(0,-48)(8,0){28}{\pspolygon[fillstyle=solid,
    fillcolor=lightgray]
  (0,0)(6,0)(3,-8)(0,-8)\pspolygon(6,0)(7,0)(4,-8)(3,-8)\pspolygon[fillstyle=solid,
    fillcolor=lightgray](7,0)(7,-8)(4,-8)\pspolygon(7,0)(8,0)(8,-8)(7,-8)}
 \end{pspicture}
\caption{Construction for $\sigma^{\eta}_{(\lambda,1)}$}
\label{fig:branch}
\end{figure}
 
With this we can define the modified phase functions:
\begin{align*}
& \tilde{\chi}_{1}^{\eta}:=   2\chi_{(\lambda,1)}(y_{2}) -1,\\
& \tilde{\chi}_{2}^{\eta}:= (1-2\mu) -\sigma^{\eta}_{(0,\lambda)}(y_{1},y_{2}) - \sigma^{\eta}_{(\lambda,1)}(y_{1},y_{2}), \\
& \tilde{\chi}_{3}^{\eta}:=  \sigma^{\eta}_{(0,\lambda)}(y_{1},y_{2}) - (1-2\mu)\chi_{(\lambda,1)}(y_{2})
 - \sigma^{\eta}_{(\lambda,1)}(y_{1},y_{2}) + (1- 2\mu)\chi_{(0,\lambda)}(y_{2}),
\end{align*}
with $\chi_{(0,\lambda)}(y_{2}):= \left\{ \begin{array}{ll} 1, & y_{2}\in (0,\lambda)\\
0, & else. \end{array} \right.
$\\
The definition of $\chi_{(\lambda,1)}$ is to be understood analogously, just as 
$\sigma^{\eta}_{(0, \lambda)}(y_{1}, y_{2})$ is defined in analogy to $\sigma^{\eta}_{(\lambda,1)}(y_{1},y_{2})$. 
We have (c.f. Figure \ref{Abb:12} (a))
\begin{align*}
 (\tilde{\chi}_{1}^{\eta}, \tilde{\chi}_{2}^{\eta}, \tilde{\chi}_{3}^{\eta}) \in \{ (1,1,1), (1,-1,-1), (-1,1,-1), (-1,-1,1) \}.
\end{align*}
This shows that $\tilde{\chi}_{1}^{\eta}, \tilde{\chi}_{2}^{\eta}, \tilde{\chi}_{3}^{\eta}$ originate from $\chi_{1}^{\eta},  \chi_{2}^{\eta},\chi_{3}^{\eta}, \chi_{4}^{\eta} \in \{0,1\}$, 
i.e. these functions are indeed modified characteristic functions.\\

\textbf{Step 4:} \itshape Energy estimates.\\
\upshape In oder to estimate the elastic energy, we first remark that via the triangle inequality $\tilde{\chi}_{1}^{\eta}, \tilde{\chi}_{2}^{\eta}, \tilde{\chi}_{3}^{\eta}$ can be decomposed into $\sigma^{\eta}_{(0,\lambda)}, \sigma^{\eta}_{(\lambda,1)}, \chi_{(0,\lambda)}, \chi_{(\lambda,1)}$.
Since $\chi_{(0,\lambda)}, \chi_{(\lambda,1)}$ only depend on $y_{2}$, the terms involving $\mathcal{F}\chi_{(0,\lambda)}, \mathcal{F}\chi_{(\lambda,1)}$ do not contribute for $k_{1}\neq 0$. 
As these expressions only occur in the definition of $\tilde{\chi}_{3}, \tilde{\chi}_{1}$ and are therefore multiplied by $k_{1}$, these functions do not play a role
for the estimate of the elastic energy.\\ 
Consequently, we can conclude:
\begin{align*}
 E_{elast}(\chi^{\eta}) 
& \lesssim \sum\limits_{\substack{ k\in \mathbb{Z}^{2}\\ k\neq 0} }|k|^{-2}k_{2}^{2}\left(|\mathcal{F}{\sigma^{\eta}}_{(\lambda,1)}(k_{1},k_{2})|^{2} + |\mathcal{F}{\sigma^{\eta}}_{(0,\lambda)}(k_{1},k_{2})|^{2}\right).
\end{align*}

\medskip

With the previous steps we can now reason as in \cite{CO08}:\\
\textbf{Step 5:} \itshape Estimate of the energy contributions originating from $\sigma^{\eta}_{(\lambda,1)}$.\\
We define
\begin{align*}
&E(\sigma^{\eta}_{(\lambda,1)}) := \eta^{-\frac{2}{3}}\sum\limits_{\substack{k\in \mathbb{Z}^{2} \\ k\neq 0}}|k|^{-2}k_{2}^{2}|\mathcal{F}{\sigma^{\eta}}_{(\lambda,1)}|^{2} + \eta^{\frac{1}{3}}\int\limits_{(0,1)^{2}}{|\nabla \sigma_{(\lambda,1)}^{\eta}|dx}\\
& =: \eta^{-\frac{2}{3}}E_{elast}(\sigma^{\eta}_{(\lambda,1)}) + \eta^{\frac{1}{3}}E_{surf}(\sigma^{\eta}_{(\lambda,1)}).
\end{align*}
\upshape \textbf{Step 5a:} \itshape Estimate of the energy within the branching region.\\
\upshape
We consider the following equivalent formulations of the elastic energy:
\begin{align*}
\MoveEqLeft{ E_{elast}(\sigma^{\eta}_{(\lambda,1)}) = \sum\limits_{\substack{k\in \mathbb{Z}^{2} \\ k\neq 0}}|k|^{-2}k_{2}^{2}|\mathcal{F}{\sigma^{\eta}}_{(\lambda,1)}|^{2}
= \int\limits_{(0,1)^{2}}{||\nabla|^{-1}\partial_{2}\sigma^{\eta}_{(\lambda,1)}|^{2}dy}}\\
& = \inf\{\int\limits_{(0,1)^{2}}|h|^{2}dy ; \; h \; (0,1)^{2}\mbox{-periodic, } \int\limits_{(0,1)^{2}}h\cdot \nabla \varphi dy = \int\limits_{(0,1)^{2}}\sigma^{\eta}_{(\lambda,1)}\partial_{2}\varphi dy \\
& \;\;\,\, \;\;  \forall \varphi: \mathbb{R}^{2} \rightarrow \mathbb{R}^{2}, \; (0,1)^{2}\mbox{-periodic}  \}.
\end{align*}
Thus, choosing $h = \begin{pmatrix} h_{1} \\ 0  \end{pmatrix}$ as indicated in the sketch (c.f. Figure \ref{fig:sigma}), we obtain an upper bound for the energy.

\begin{figure}[hbtp]
\begin{pspicture}(0,-1)(9,5)
\pspolygon(0,0)(4,0)(4,4)(0,4)
\pspolygon(0,0)(2,0)(1,4)(0,4)
\pspolygon(3,0)(3,4)(2,4) 
\put(.25,3){+1}
\put(1.5,3){-1}
\put(2.4,3){+1}
\put(3.5,3){-1}
\psline{|*-|*}(0,-0.4)(4,-.4)
\put(2,-.7){$w_{n}$}
\psline{|*-|*}(0,4.4)(2,4.4)
\put(0.7,4.6){$w_{n+1}$}
\psline{|*-|*}(4.4,0)(4.4,4)
\put(4.6,2){$l_{n}$}

\pspolygon(6,0)(10,0)(10,4)(6,4)
\pspolygon(8,0)(9,0)(8,4)(7,4)
\psline{->}(7.6,2)(8.4,2)
\psline{->}(7.85,1)(8.65,1)
\psline{->}(7.35,3)(8.15,3)
\put(6.5,1){$h=0$}
\put(8.5,3){$h=0$}

\end{pspicture}
\caption{$\sigma^{\eta}_{(\lambda,1)}$ and the corresponding ``field'' $h$}
\label{fig:sigma}
\end{figure}
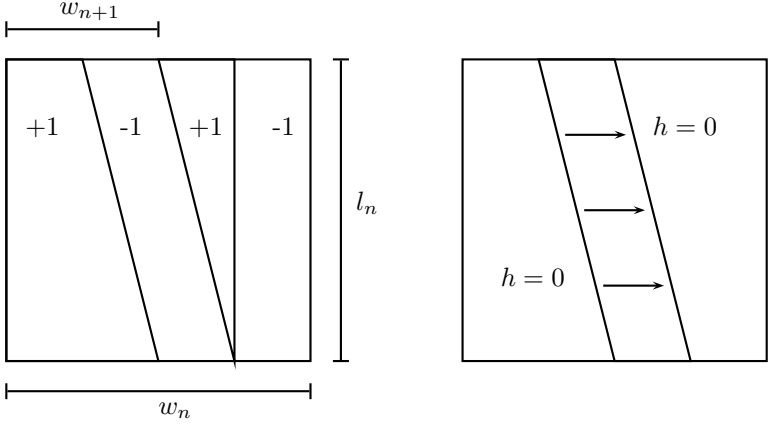

Since
$
 |h| = |h_{1}| \lesssim w_{n}l_{n}^{-1}
$
and choosing $w_{n} \leq l_{n}$, we have:
\begin{align*}
 \eta^{\frac{1}{3}}\int|\nabla \sigma^{\eta}_{(\lambda,1)}|dy + \eta^{-\frac{2}{3}}\int|h|^{2}dy
& \lesssim \eta^{\frac{1}{3}}\max\{w_{n}, l_{n}\} + \eta^{-\frac{2}{3}}(w_{n}l_{n}^{-1})^{2}(w_{n}l_{n})\\
& \lesssim \eta^{\frac{1}{3}}l_{n} + \eta^{-\frac{2}{3}}w_{n}^{3}l_{n}^{-1}.
\end{align*}
Summing over the $w_{n}^{-1}$-cells per row, the total energy is bounded by:
\begin{align*}
\sum\limits_{n=1}^{N}w_{n}^{-1}(\eta^{\frac{1}{3}}l_{n} + \eta^{-\frac{2}{3}}w_{n}^{3}l_{n}^{-1}) =  \sum\limits_{n=1}^{N}(w_{n}^{-1}\eta^{\frac{1}{3}}l_{n} + \eta^{-\frac{2}{3}}w_{n}^{2}l_{n}^{-1}).
\end{align*}

\textbf{Step 5b:} \itshape Estimate of the energy within the transition layers.\\
\upshape On $(0,1)\times\{y_{2}= \lambda\}$, $\sigma^{\eta}_{(\lambda,1)}$ jumps from $1-2\mu$ to $1$ and $-1$ respectively; on $(0,1)\times\{y_{2}= 1\}$ the situation is analogous. 
Thus, $\partial_{2}\sigma^{\eta}_{(\lambda,1)}$ displays a singular contribution on these lines. In order to compensate these ``charges'' we construct a ``field'' $h$.
For this purpose we introduce the transition layers
\begin{align*}
 (0,1)\times\left(\lambda - w_{N+1}, \lambda \right) \cup (0,1)\times(1,1+w_{N+1}),
\end{align*}
where $w_{N+1}\leq \frac{\lambda}{2}$. In oder to secure the permissibility of the ``field'' $h$ belonging to the $H^{-1}$-norm we have to account for the jumps in direction of the normal.
The ``field'' $h$ is constructed in cells of width $w_{N+1}$ in the first transition layer via a potential $u$ with $h=-\nabla u$. In other words, $u$ is prescribed to satisfy\footnote{ At
this point the correct volume fractions for the phases play an essential role as this guarantees that the Neumann problem for the Laplacian can be solved.} 
\begin{align*}
& - \Delta u = 0 \;\; \mbox{ in } (0, w_{N+1})\times\left(\lambda -w_{N+1},\lambda \right)\\
& \frac{\partial u}{\partial \nu} = \left\{ \begin{array}{ll}
                                             2\mu & y_{1}\in (0, (1-\mu) w_{N+1}), \; y_{2} = \lambda\\
					     -2(1-\mu) & y_{1}\in ((1-\mu) w_{N+1}, w_{N+1}), \, y_{2} = \lambda\\
					     0 & \mbox{else}.
                                            \end{array}\right.
\end{align*}
For the second transition layer a similar construction is employed.
Thus, we can estimate the elastic energy in the transition layer:
\begin{align*}
 \int\limits_{(0,w_{N+1})\times (\lambda - w_{N+1}, \lambda)}|h|^{2}dy
 = \int|\nabla u|^{2}dy \lesssim w_{N+1}^{2} 
\end{align*}
for each of the $w_{N+1}^{-1}$ cells.

\textbf{Step 5c:} \itshape Combination of the energy estimates.\\
\upshape Combining the observations of step 5a and step 5b, we obtain:
\begin{align*}
E_{\eta}(\sigma_{(\lambda,1)}^{\eta}) & = \eta^{\frac{1}{3}}\int|\nabla \sigma^{\eta}_{(\lambda,1)}|dy + \eta^{-\frac{2}{3}}\int|h|^{2}dy \\
& \lesssim \sum\limits_{n=1}^{N}(\eta^{\frac{1}{3}}w_{n}^{-1}l_{n} + \eta^{-\frac{2}{3}}w_{n}^{2}l_{n}^{-1}) + \eta^{\frac{1}{3}} + \eta^{-\frac{2}{3}}w_{N+1}\\
& = \sum\limits_{n=1}^{N}(\eta^{\frac{1}{3}}2^{(n-1)(1-\beta)}\frac{l_{1}}{w_{1}} + \eta^{-\frac{2}{3}}2^{(n-1)(\beta -2)}\frac{w_{1}^{2}}{l_{1}})\\
& \;\;\;\;\;\; + \eta^{\frac{1}{3}} + 2^{-(N+1)}\eta^{-\frac{2}{3}} w_{1}\\
& \lesssim \eta^{\frac{1}{3}}\frac{l_{1}}{w_{1}} + \eta^{-\frac{2}{3}}\frac{w_{1}^{2}}{l_{1}} + \eta^{\frac{1}{3}} + 2^{-(N+1)}\eta^{-\frac{2}{3}} w_{1},
\end{align*}
where we used that, for instance for $\beta= \frac{3}{2}$, we have
\begin{align*}
& \sum\limits_{n=1}^{N}2^{(n-1)(1-\beta)} \leq C, \;\; \sum\limits_{n=1}^{N}2^{(n-1)(\beta-2)} \leq C.
\end{align*}
Recalling $\frac{1-\lambda}{2} \sim l_{1}$ and minimizing $\eta^{\frac{1}{3}}\frac{1-\lambda}{w_{1}} + \eta^{-\frac{2}{3}}\frac{w_{1}^{2}}{1-\lambda}$ with respect to $w_{1}$, 
we obtain 
\begin{align*}
w_{1}= \eta^{\frac{1}{3}}(1-\lambda) \mbox{ and } E_{\eta}(\chi^{\eta}) \lesssim 1 +\eta^{\frac{1}{3}} + 2^{-(N+1)}\eta^{-\frac{1}{3}} \lesssim 1,
\end{align*}
if we choose $N$ so large that $2^{-(N+1)}\lesssim \eta^{\frac{1}{3}}$. 
Additionally, we have to choose $N$ as to satisfy 
\begin{align*}
& w_{n}\leq l_{n},\\
& w_{N+1} \leq \frac{\lambda}{2}.
\end{align*}
For the first condition we define $N$ such that $\eta^{\frac{2}{3}}\leq 2^{-N}$ which can be achieved as $\eta \ll 1$ (which due to $\eta \leq 1$ does not contradict $2^{-(N+1)} \lesssim \eta^{\frac{1}{3}}$). In particular, we can require $\eta^{\frac{2}{3}} \sim 2^{-N} $. 
Thus, the second condition reduces to $(1-\lambda) \eta \leq \lambda$, which can also be satisfied due to $\eta \ll \lambda$.

\textbf{Step 6:} \itshape Conclusion.\\
\upshape Since $\sigma^{\eta}_{(0,\lambda)}$ and $\sigma^{\eta}_{(\lambda,1)}$ display the same scaling behavior the previous computations lead to the overall energy estimate:
\begin{align*}
\MoveEqLeft{ E_{\eta}(\chi^{\eta})  = \eta^{\frac{1}{3}}\int|\nabla \chi_{1}^{\eta}| +|\nabla \chi_{2}^{\eta}|+|\nabla \chi_{3}^{\eta}|+|\nabla \chi_{4}^{\eta}| dx + \eta^{-\frac{2}{3}}E_{elast}} \\
& \lesssim E(\sigma^{\eta}_{(0,\lambda)}) +  E(\sigma^{\eta}_{(\lambda,1)}) + \eta^{\frac{1}{3}}\int\limits_{(0,1)^{2}}{|\nabla \chi_{\{y_{2}\in(\lambda,1)\}}|dx} + \eta^{\frac{1}{3}}\int\limits_{(0,1)^{2}}{|\nabla \chi_{\{y_{2}\in(0,\lambda)\}}|dx}\\
& \lesssim 1 + \eta^{\frac{1}{3}}\lesssim 1.
\end{align*}
This proves the upper bound.\\
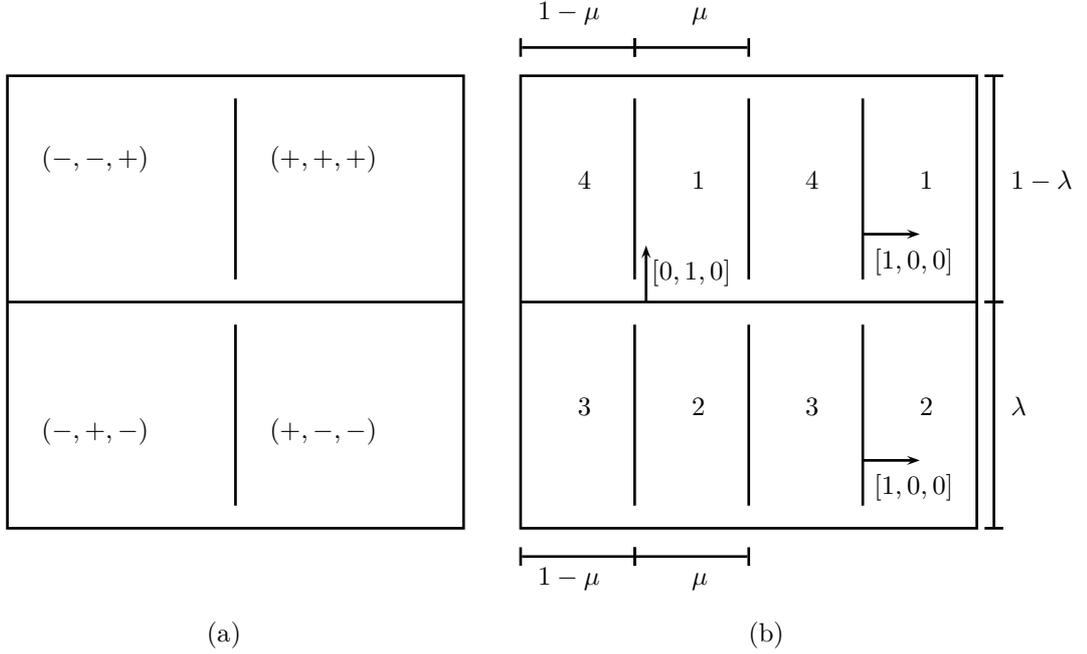
\begin{figure}[hbt]

\begin{pspicture}(0.5,-1.5)(10,7)
\psset{unit=1.5cm}
\psset{linewidth=1pt}

\pspolygon(0.5,0)(4.5,0)(4.5,4)(0.5,4)
\psline(0.5,2)(4.5,2)
\psline(2.5,0.2)(2.5,1.8)
\psline(2.5,2.2)(2.5,3.8)
\put(0.8,0.8){$(-,+,-)$}
\put(2.8,0.8){$(+,-,-)$}
\put(0.8,3.2){$(-,-,+)$}
\put(2.8,3.2){$(+,+,+)$}

\put(2.25,-1){(a)}

\pspolygon(5,0)(9,0)(9,4)(5,4)
\psline(5,2)(9,2)
\psline(6,0.2)(6,1.8)
\psline(7,0.2)(7,1.8)
\psline(8,0.2)(8,1.8)
\psline(6,2.2)(6,3.8)
\psline(7,2.2)(7,3.8)
\psline(8,2.2)(8,3.8)
\put(5.5,1){3}
\put(6.5,1){2}
\put(7.5,1){3}
\put(8.5,1){2}
\put(5.5,3){4}
\put(6.5,3){1}
\put(7.5,3){4}
\put(8.5,3){1}
\psline{|*-|*}(5,-0.25)(6,-0.25)
\put(5.15,-0.5){$1-\mu$}
\psline{|*-|*}(6,-.25)(7,-.25)
\put(6.5,-.5){$\mu$}
\psline{|*-|*}(9.15,0)(9.15,2)
\put(9.3,1){$\lambda$}
\psline{|*-|*}(9.15,2)(9.15,4)
\put(9.3,3){$1-\lambda$}
\psline{->}(6.1,2)(6.1,2.5)
\put(6.15,2.2){$[0,1,0]$}

\psline{->}(8,2.6)(8.5,2.6)
\put(8.1,2.3){$[1,0,0]$}

\psline{->}(8,.6)(8.5,.6)
\put(8.1,.3){$[1,0,0]$}

\psline{|*-|*}(5,4.25)(6,4.25)
\put(5.15,4.5){$1-\mu$}
\psline{|*-|*}(6,4.25)(7,4.25)
\put(6.5,4.5){$\mu$}

\put(7,-1){(b)}

\end{pspicture}
\caption{Schematic arrangement of (a) the modified phase functions $(\tilde{\chi}_{2}, \tilde{\chi}_{3}, \tilde{\chi}_{1})$, (b) the actual phase functions $(\chi_{1},\chi_{2},\chi_{3},\chi_{4})$}
\label{Abb:12}
\end{figure}

In order to verify (\ref{eq:Mikr1}), (\ref{eq:Mikr2}), we calculate $\theta_{i}^{\eta}$ as functions of $\lambda$ and of $\mu$.
Due to the choice of the phases $\chi_{1},...,\chi_{4}$ as functions of $\sigma^{\eta}, \chi$ we can easily determine the volume fractions (c.f. Figure \ref{Abb:12}):  
\begin{align*}
&\chi_{1}^{\eta} = \frac{1}{2}(\chi_{(\lambda,1)}(y_{2}) - \sigma^{\eta}_{(\lambda,1)}(y_{1},y_{2}) + (1-2\mu)\chi_{(0,\lambda)}(y_{2})),\\
&\chi_{2}^{\eta} = \frac{1}{2}(1 - \chi_{(\lambda,1)}(y_{2}) - \sigma^{\eta}_{(0,\lambda)}(y_{1},y_{2}) + (1-2\mu)\chi_{(\lambda,1)}(y_{2})),\\
&\chi_{3}^{\eta} = \frac{1}{2}(1 + \sigma^{\eta}_{(0,\lambda)}(y_{1},y_{2}) - (1-2\mu)\chi_{(\lambda,1)}(y_{1},y_{2}) - \chi_{(\lambda,1)}(y_{2})),\\
&\chi_{4}^{\eta} = \frac{1}{2}(\sigma^{\eta}_{(\lambda,1)}(y_{1},y_{2}) - (1-2\mu)\chi_{(0,\lambda)}(y_{2}) + \chi_{(\lambda,1)}(y_{2})).
\end{align*}
Thus, we find
\begin{align*}
& \theta^{\eta}_{1}=\mu(1 - \lambda), \; \theta^{\eta}_{2} = \mu\lambda, \; \theta^{\eta}_{3} = (1-\mu)\lambda, \; \theta^{\eta}_{4} = (1-\mu)(1-\lambda).
\end{align*}
Finally, this yields
\begin{align*}
& \theta^{\eta}_{1}(\theta^{\eta}_{2} + \theta^{\eta}_{4}) - \theta^{\eta}_{2}(\theta^{\eta}_{1} + \theta^{\eta}_{3}) = (1-\mu)\mu(1-2\lambda),\\
& \theta^{\eta}_{4}(\theta^{\eta}_{1} + \theta^{\eta}_{3}) - \theta^{\eta}_{1}(\theta^{\eta}_{2} + \theta^{\eta}_{4}) = \lambda(1-\lambda)(1-2\mu). 
\end{align*}
Choosing $\lambda = \mu$, the right hand expression takes values in $(0, \frac{3}{16}]$ for $\mu\in(0,1)$. As the energy estimate was independent of 
$\mu, \; \lambda$, we can choose $\mu,\; \lambda$ arbitrarily. This implies the claim.
\end{proof}

\end{document}